\documentclass[12pt]{amsart}

\oddsidemargin \evensidemargin
\usepackage{color}
\usepackage{amsmath, latexsym, amssymb}
\usepackage{graphicx}
\usepackage{amsfonts}
\usepackage{subfigure}
\usepackage{amsthm}
\usepackage{amsmath}
\usepackage{graphics}
\usepackage{amssymb}
\usepackage{latexsym}
\usepackage{amscd}
\usepackage{color}
\usepackage{graphicx}
\usepackage{subfigure}
\usepackage{pict2e}
\usepackage[all]{xy}
\usepackage{bm}
\usepackage{enumerate}

\newtheorem{theorem}{Theorem}
\newtheorem{proposition}{Proposition}

\newtheorem{lemma}{Lemma}
\newtheorem{cor}[theorem]{Corollary}
\newtheorem{procedure}[theorem]{Procedure}
\theoremstyle{definition}
\newtheorem{definition}{Definition}
\newtheorem{algorithm}[theorem]{Algorithm}
\newtheorem*{acknowledgements*}{Acknowledgements}
\numberwithin{equation}{section}

\DeclareMathOperator{\SSYRT}{SSYRT}
\DeclareMathOperator{\SSRRT}{SSRRT}
\DeclareMathOperator{\SSYCT}{SSYCT}
\DeclareMathOperator{\SYRT}{SYRT}

\DeclareMathOperator{\SSYT}{SSYT}

\DeclareMathOperator{\col}{col}

\newcommand{\YL}{\mathcal{L}}
\DeclareMathOperator{\shape}{\bm \lambda}
\newcommand{\RSYQS}{row-strict Young quasisymmetric Schur function}
\DeclareMathOperator{\comp}{comp}

\newlength\cellsize 
\savebox2{%
\begin{picture}(15,15)
\put(0,0){\line(1,0){15}}
\put(0,0){\line(0,1){15}}
\put(15,0){\line(0,1){15}}
\put(0,15){\line(1,0){15}}
\end{picture}}
\newcommand\cellify[1]{\def\thearg{#1}\def\nothing{}%
\ifx\thearg\nothing
\vrule width0pt height\cellsize depth0pt\else
\hbox to 0pt{\usebox2\hss}\fi%
\vbox to 15\unitlength{
\vss
\hbox to 15\unitlength{\hss$#1$\hss}
\vss}}
\newcommand\tableau[1]{\vtop{\let\\=\cr
\setlength\baselineskip{-16000pt}
\setlength\lineskiplimit{16000pt}
\setlength\lineskip{0pt}
\halign{&\cellify{##}\cr#1\crcr}}}
\savebox3{%
\begin{picture}(15,15)
\put(0,0){\line(1,0){15}}
\put(0,0){\line(0,1){15}}
\put(15,0){\line(0,1){15}}
\put(0,15){\line(1,0){15}}
\end{picture}}
\newcommand\expath[1]{%
\hbox to 0pt{\usebox3\hss}%
\vbox to 15\unitlength{
\vss
\hbox to 15\unitlength{\hss$#1$\hss}
\vss}}

\title{Skew row-strict quasisymmetric Schur functions } 

\author{Sarah K. Mason     \and
        Elizabeth Niese 
}

\begin{document}


\maketitle

\begin{abstract}
Mason and Remmel introduced a basis for quasisymmetric functions known as the row-strict quasisymmetric Schur functions.  This basis is generated combinatorially by fillings of composition diagrams that are analogous to the row-strict tableaux that generate Schur functions.  We introduce a modification known as Young row-strict quasisymmetric Schur functions, which are generated by row-strict Young composition fillings.  After discussing basic combinatorial properties of these functions, we define a skew Young row-strict quasisymmetric Schur function using the Hopf algebra of quasisymmetric functions and then prove this is equivalent to a combinatorial description.  We also provide a decomposition of the skew Young row-strict quasisymmetric Schur functions into a sum of Gessel's fundamental quasisymmetric functions and prove a multiplication rule for  the product of a Young row-strict quasisymmetric Schur function and a Schur function.
\end{abstract}

\section{Introduction}\label{sec:intro}

Schur functions are symmetric polynomials introduced by Schur~\cite{Sch01} as characters of irreducible representations of the general linear group of invertible matrices.  The Schur functions can be generated combinatorially using semi-standard Young tableaux and form a basis for the ring, $Sym$, of symmetric functions.  The product of two Schur functions has positive coefficients when expressed in terms of the Schur function basis, where the coefficients are given by a combinatorial rule called the Littlewood-Richardson Rule~\cite{Ful97}.  These Littlewood-Richardson coefficients appear in algebraic geometry in the cohomology of the Grassmannian.  They also arise as the coefficients when skew Schur functions are expressed in terms of ordinary Schur functions.

Skew Schur functions generalize Schur functions and are themselves of interest in discrete geometry, representation theory, and mathematical physics as well as combinatorics.  They can be generated by fillings of skew diagrams or by a generalized Jacobi-Trudi determinant.  The inner products of certain skew Schur functions are enumerated by collections of permutations with certain descent sets.  In representation theory, character formulas for certain highest weight modules of the general linear Lie algebra are given in terms of skew Schur functions~\cite{Ugl00}.

The ring $Sym$ generalizes to a ring of nonsymmetric functions $QSym$, which is a Hopf algebra dual to the noncommutative symmetric functions $NSym$.  The ring of quasisymmetric functions was introduced by Gessel~\cite{Ges84} as a source for generating functions for $P$-partitions and has since been shown to be the terminal object in the category of combinatorial Hopf algebras~\cite{ABS06}.  This ring is also the dual of the Solomon descent algebra~\cite{MalReu95} and plays an important role in permutation enumeration~\cite{GesReu93} and reduced decompositions in finite Coxeter groups~\cite{Sta84}.  Quasisymmetric functions also arise as characters of certain degenerate quantum groups~\cite{Hiv00} and form a natural setting for many enumeration problems~\cite{Sta99}. 

A new basis for quasisymmetric functions called the ``quasisymmetric Schur functions" arose through the combinatorial theory of Macdonald polynomials by summing all Demazure atoms (non-symmetric Macdonald polynomials specialized at $q=t=0$) whose indexing weak composition reduces to the same composition when its zeros are removed~\cite{HLMvW09}.  Elements of this basis are generating functions for certain types of fillings of composition diagrams with positive integers such that the row entries of these fillings are weakly decreasing.  Reversing the entries in these fillings creates fillings of composition diagrams with positive integers that map more naturally to semi-standard Young tableaux, allowing for a more direct connection to classical results for Schur functions.  These fillings are used to generate the Young quasisymmetric Schur functions introduced in~\cite{LMvW13}.

Row-strict quasisymmetric Schur functions were introduced by Mason and Remmel~\cite{MasRem10} in order to extend the duality between column- and row-strict tableaux to composition tableaux.  When given a column-strict tableau, a row-strict tableau can be obtained by taking the conjugate (reflecting over the main diagonal).  Since quasisymmetric Schur functions are indexed by compositions, conjugation is not as straightforward.  Reflecting over the main diagonal may not result in a composition diagram.  Mason and Remmel introduced a conjugation-like operation on composition diagrams that takes a column-strict composition tableau with underlying partition diagram $\lambda$ to a row-strict composition tableau with underlying partition diagram $\lambda'$, the conjugate of $\lambda$.  




In this paper, we introduce Young row-strict quasisymmetric Schur functions, which are conjugate to the Young quasisymmetric Schur functions under the omega operator in the same way that the row-strict quasisymmetric Schur functions are conjugate to the quasisymmetric Schur functions.  This new basis can be expanded positively in terms of the fundamental quasisymmetric functions, admits a Schensted type of insertion, and has a multiplication rule which refines the Littlewood-Richardson rule.  These properties are similar to properties of the row-strict quasisymmetric Schur functions explored by Mason and Remmel~\cite{MasRem10}, but we introduce a skew version of the Young row-strict quasisymmetric Schur functions, which is not known for row-strict quasisymmetric Schur functions.  The objects used to generate the Young row-strict quasisymmetric Schur functions are in bijection with transposes of semi-standard Young tableaux, and therefore these new functions fit into the bigger picture of bases for quasisymmetric functions in a natural way.  In fact, the objects introduced in this paper are the missing piece when rounding out the quasisymmetric story that is analogous to the Schur functions in the ring of symmetric functions.  Throughout the paper we show the connections between the Young row-strict quasisymmetric Schur functions and the row-strict quasisymmetric Schur functions in~\cite{MasRem10}. 


In Section~\ref{sec:back} we review the background on Schur functions and quasisymmetric Schur functions and introduce Young row-strict composition diagrams.  In Section~\ref{sec:rsqf} we define the Young row-strict quasisymmetric Schur functions and prove several fundamental results about the functions.  In Section~\ref{sec:skewrs} we define skew Young row-strict quasisymmetric Schur functions via Hopf algebras and then provide a combinatorial definition of the functions.  Finally, in Section~\ref{sec:LRrule}, we prove an analogue to the Littlewood-Richardson rule and an analogue to conjugation for Young composition tableaux.

\section{Background}\label{sec:back}

\subsection{Partitions and Schur functions}

A {\it partition} $\mu$ of a positive integer $n$, written $\mu \vdash n$,  is a finite, weakly decreasing sequence of positive integers, $\mu=(\mu_1,\mu_2,\ldots, \mu_k)$, such that $\sum \mu_i = n$.  Each $\mu_i$ is a {\it part} of $\mu$ and $n$ is the {\it weight} of the partition.  The {\it length} of the partition is $k$, the number of parts in the partition.  Given a partition $\mu = (\mu_1,\mu_2, \ldots, \mu_k)$, the {\it diagram} of $\mu$ is the collection of left-justified boxes (called {\it cells}) such that there are $\mu_i$ boxes in the $i$th row from the bottom.  This is known as the {\it French convention} for the diagram of a partition.  (The English convention places $\mu_i$ left-justified boxes in the $i$th row from the {\it top}.)  Here, the cell label $(i,j)$ refers to the cell in the $i^{th}$ row from the bottom and the $j^{th}$ column from the left.  Given partitions $\lambda$ and $\mu$ of $n$, we say that $\lambda \geq \mu$ in {\it dominance order} if $\lambda_1+\lambda_2+\cdots+\lambda_i \geq \mu_1+\mu_2+\cdots+\mu_i$ for all $i\geq 1$ with $\lambda>\mu$ if $\lambda\geq \mu$ and $\lambda\neq \mu$.

A {\it semi-standard Young tableau} (SSYT) of shape $\mu$ is a placement of positive integers into the diagram of $\mu$ such that the entries are strictly increasing up columns from bottom to top and weakly increasing across rows from left to right.  Let $\SSYT(\mu,k)$ denote the set of SSYT of shape $\mu$ filled with labels from the set $[k] = \{1,2, \ldots, k\}$. The {\it weight} of a SSYT $T$ is $x^T=\prod_i x_i^{v_i}$ where $v_i$ is the number of times $i$ appears in $T$.  We can now define a {\it Schur function} as a generating function of semi-standard Young tableaux of a fixed shape.  
\begin{definition}\label{def:schur}
Let $\mu$ be a partition of $n$.  Then, 
\[s_\mu(x_1,x_2,\ldots, x_k) = \sum_{T \in \SSYT(\mu,k)}x^T.\]
\end{definition}

It is known that $\{s_\mu:\mu \vdash n\}$ is a basis for the space $Sym_n$ of symmetric functions of degree $n$.  We can also define a {\it row-strict semi-standard Young tableau} of shape $\mu$ as a placement of positive integers into the diagram of $\mu$ such that the entries are weakly increasing up columns and strictly increasing across rows from left to right.  Note that if $T$ is a row-strict semi-standard Young tableau, the conjugate of $T$, denoted $T'$, obtained by reflecting $T$ over the line $y=x$, is a SSYT, as shown in Fig. \ref{fig:RSSYT}.  Reading order for a row-strict semi-standard tableau is up each column, starting from right to left. 

\begin{figure}
\[ T = \tableau{4\\2&3&5\\1&2&5\\1&2&4&5&6} \qquad\qquad T' = \tableau{6\\5\\4&5&5\\2&2&3\\1&1&2&4}\]
\caption{A row-strict semi-standard Young tableau $T$ of shape $(5,3,3,1)$ and its conjugate $T'$.}\label{fig:RSSYT}
\end{figure}

A {\em standard Young tableau} (SYT) of shape $\mu$ a partition of $n$ is a placement of $1,2,\ldots,n$, each exactly once, into the diagram of $\mu$ such that each column increases from bottom to top and each row increases from left to right.  We can find the {\em standardization} $std(T)$ of a row-strict semi-standard tableau $T$ by replacing the $k_1$ ones with $1,2,\ldots, k_1$, in reading order, then replacing the $k_2$ twos with $k_1+1,\ldots, k_1+k_2$, and so on.  The results of standardization are seen in Fig.~\ref{fig:stdz}. 

\begin{figure} 
\[ T = \tableau{2\\2&4&5\\1&2&4&5\\1&2&3&5} \qquad\qquad std(T) = \tableau{6\\5&9&12\\2&4&8&11\\1&3&7&10}\]
\caption{Standardization of a row-strict Young tableau $T$.}\label{fig:stdz}
\end{figure}  

\subsection{Compositions and quasisymmetric Schur functions}

A {\it weak composition} of $n$ is a finite sequence of nonnegative integers that sum to $n$.  The individual integers appearing in such a sequence are called its {\it parts}, and $n$ is referred to as the {\it weight} of the weak composition.  A {\it strong composition} (often just called a {\it composition}) $\alpha$ of $n$, written $\alpha \vDash n$, is a finite sequence of positive integers that sum to $n$.  If $\alpha=(\alpha_1, \alpha_2, \hdots , \alpha_k)$, then $\alpha_i$ is a {\it part} of $\alpha$ and $| \alpha| = \sum_{i=1}^k \alpha_i$ is the {\it weight} of $\alpha$.  The {\it length} of a composition (or weak composition), denoted $\ell(\alpha)$, is the number of parts in the composition.  If the parts of a composition $\alpha$ can be rearranged into a partition $\lambda$, we say that the underlying partition of $\alpha$ is $\lambda$ and write $\shape(\alpha) = \lambda$.  A composition $\beta$ is said to be a {\it refinement} of a composition $\alpha$ if $\alpha$ can be obtained from $\beta$ by summing collections of consecutive parts of $\beta$.

Given a composition $\alpha=(\alpha_1, \hdots , \alpha_k)$, its {\it diagram} is given by placing boxes (or {\it cells}) into left-justified rows so that the $i^{th}$ row from the bottom contains $\alpha_i$ cells.  Note that this is analogous to the French notation for the Young diagram of a partition.  

There is a bijection between subsets of $[n-1]$ and compositions of $n$.  Suppose $S=\{s_1,s_2,\ldots,s_k\} \subseteq [n-1]$.  Then the corresponding composition is $\comp(S) = (s_1,s_2-s_1, \ldots, s_k-s_{k-1},n-s_k)$.  Similarly, given a composition $\alpha = (\alpha_1, \alpha_2, \ldots, \alpha_k)$, the corresponding subset $S = \{s_1, s_2, \ldots, s_{k-1}\}$ of $[n-1]$ is obtained by setting $s_i = \sum_{j=1}^i \alpha_j$.  We also need the notion of {\it complementary compositions}.  The {\it complement $\tilde{\beta}$} to a composition $\beta=\comp(S)$ arising from a subset $S \subseteq [n-1]$ is the composition obtained from the subset $S^c = [n-1]-S$.  For example, the composition $\beta = (1,4,2)$ arising from the subset $S=\{1,5 \} \subseteq [6]$ has complement $\tilde{\beta}=(2,1,1,2,1)$ arising from the subset $S^c=\{2,3,4,6\}$.

A {\it quasisymmetric function} is a bounded degree formal power series $f(x) \in \mathbb{Q}[[x_1, x_2, \hdots ]]$ such that for all compositions $\alpha=(\alpha_1, \alpha_2, \hdots, \alpha_k)$, the coefficient of $\prod x_i^{\alpha_i}$ is equal to the coefficient of $\prod x_{i_j}^{\alpha_i}$ for all $i_1 < i_2 < \hdots < i_k$.  Let $QSym$ denote the ring of quasisymmetric functions 
and $QSym_n$ denote the space of homogeneous quasisymmetric functions 
of degree $n$ so that $\displaystyle{QSym = \oplus_{n \geq 0} QSym_n}$. 

A natural basis for $QSym_n$  is the {\it monomial quasisymmetric basis}, given by the set of all $M_\alpha$ such that  $\alpha \vDash n$ where  
$$M_{\alpha} = \sum_{i_1 < i_2 < \cdots < i_k} x_{i_1}^{\alpha_1} x_{i_2}^{\alpha_2} \cdots x_{i_k}^{\alpha_k}.$$  
Gessel's {\it fundamental basis for quasisymmetric functions}~\cite{Ges84} can be expressed by $$F_{\alpha} = \sum_{\beta \preceq \alpha} M_{\beta},$$ where $\beta \preceq \alpha$ means that $\beta$ is a refinement of $\alpha$.

\subsection{Partial orderings of compositions}

The Young composition poset $\YL_c$~\cite{LMvW13} is the poset consisting of all (strong) compositions with $\alpha = (\alpha_1,\alpha_2,\ldots,\alpha_k)$ covered by 
\begin{enumerate}
\item  $(\alpha_1,\alpha_2,\ldots, \alpha_k, 1)$, that is, the composition formed by appending a part of size 1 to $\alpha$, and
\item $\beta = (\alpha_1,\ldots,\alpha_j+1, \ldots,\alpha_k)$ with $\alpha_i\neq \alpha_j$ for all $i>j$. That is, $\beta$ is the composition formed from $\alpha$ by adding one to the rightmost part of any given size.
\end{enumerate}

The composition $\alpha$ is said to be {\it contained} in the composition $\beta$ (written $\alpha \subset \beta$) if and only if $\ell(\alpha) \le \ell(\beta)$ and $\alpha_i \le \beta_i$ for all $1 \le i \le \ell(\alpha)$.  Note that if $\alpha<_{\YL_c}\beta$, then $\alpha \subset \beta$, though the converse is not true.  Given $\alpha \subset \beta$, we define a {\it skew composition diagram} of shape $\beta//\alpha$ to be the diagram of $\beta$ with the cells of $\alpha$ removed from the bottom left corner, as seen in Fig. \ref{fig:skew}.  
\begin{figure}
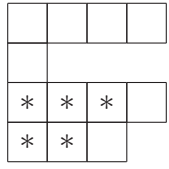

\[\tableau{{}&{}&{}&{}\\{}\\*&*&*&{}\\*&*&{}}\]
\caption{A skew composition diagram of shape $(3,4,1,4)//(2,3)$. }\label{fig:skew}
\end{figure}

\subsection{Composition tableaux}\label{sec:comptab}

Row-strict quasisymmetric Schur functions were introduced by Mason and Remmel \cite{MasRem10} as weighted sums of what we will call {\em semi-standard  reverse row-strict composition tableaux} (SSRRT).  Given a composition $\alpha=(\alpha_1, \alpha_2,\ldots, \alpha_k)$ with diagram written in the English convention (that is, $\alpha_1$ is the length of the top row, $\alpha_2$ the length of the next row down, etc.), a  SSRRT $T$ is a filling of the cells of $\alpha$ with positive integers such that 
\begin{enumerate}
\item row entries are strictly decreasing from left to right,
\item the entries in the leftmost column are weakly increasing from top to bottom, and 
\item (triple condition) for $1\leq i<j\leq \ell(\alpha)$ and $1\leq k < m$, where $m$ is the size of the largest part of $\alpha$, if $T(j,k+1)<T(i,k )$, then $T(j,k+1)\leq T(i,k+1)$, assuming that the entry in any cell not contained in $\alpha$ is 0, where $T(a,b)$ is the entry in row $a$, column $b$.  
\end{enumerate}
For our purposes, we instead introduce {\em Young row-strict composition tableaux}, which are fillings of $\alpha//\beta$ where the diagram of $\alpha//\beta$ is written in the French convention.  See Appendix~\ref{appendix} for a table of the various types of fillings of composition diagrams appearing throughout this paper.

\begin{definition}\label{SSYRTdef}
A filling $T: \alpha // \beta \rightarrow \mathbb{Z}_+$ is a {\it semi-standard Young row-strict composition tableau (SSYRT)} of shape $\alpha // \beta$ if it satisfies the following conditions:

\begin{enumerate}
\item row entries are strictly increasing from left to right,
\item the entries in the leftmost column are weakly decreasing from top to bottom, and 
\item (triple condition) for $1\leq i<j \leq \ell(\alpha)$ and $1\leq k<m$, where $m$ is the size of the largest part of $\alpha$, if $T(j,k)<T(i,k+1)$, then $T(j,k+1)\leq T(i,k+1)$, assuming the entry in any cell not contained in $\alpha$ is $\infty$ and the entry in any cell contained in $\beta$ is 0.  The arrangement of these cells can be seen in Fig.~\ref{fig:triplecond}.
\end{enumerate}
\end{definition}

\begin{figure}
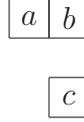

\[ \tableau{a&b\\\\&c}\]
\caption{A triple of cells from Def.~\ref{SSYRTdef} where $T(j,k)=a$, $T(j,k+1)=b$, and $T(i,k+1)=c$.  This triple satisfies the triple condition if $a<c$ implies $b\leq c$.}\label{fig:triplecond}
\end{figure}

\begin{figure}
$$T=\tableau{1 & 2 & 3 & 4 \\ 1 \\ * & * & * & 1 \\ * & * & 4}$$
\caption{An SSYRT of shape $(3,4,1,4) // (2,3)$ and weight $x^T=x_1^3 x_2 x_3 x_4^2$.}\label{fig:SSYRT}
\end{figure}

Note that the triple condition guarantees that if $\beta \subset \alpha$ but $\beta \nless_{\YL_c} \alpha$, then there are no SSYRT of shape $\alpha//\beta$.  The {\em weight} of a SSYRT $T$ of shape $\alpha// \beta$ is the monomial $x^T=\prod_i x_i^{v_i}$ where $v_i$ is the number of times the label $i$ appears in $T$ as seen in Fig.~\ref{fig:SSYRT}.

There is a diagram-reversing, weight reversing bijection, $f$, between \\$\SSRRT(\alpha, m)$ and $\SSYRT(rev(\alpha), m)$ where $rev(\alpha)=(\alpha_k,\alpha_{k-1}, \ldots, \alpha_1)$. \\ Note that since the diagrams for  SSRRTs follow the English convention and the diagrams for SSYRTs follow the French convention, the shapes resulting from this bijection will appear to be the same, but correspond to two distinct compositions.  To obtain the entries in the new filling, replace each entry $i$ with $m-i+1$ as shown in Fig.~\ref{fig:f}.  It is routine to show that the triple conditions are met by this bijection.

\begin{figure}
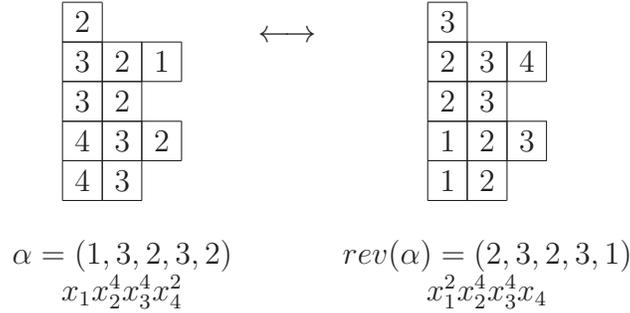

\[ \begin{array}{ccc}\tableau{2\\3&2&1\\3&2\\4&3&2\\4&3} &\longleftrightarrow & \tableau{  3\\2&3&4\\2&3\\1&2&3\\1&2}\\\\\alpha=(1,3,2,3,2) && rev(\alpha)=(2,3,2,3,1)\\x_1x_2^4x_3^4x_4^2&&x_1^2x_2^4x_3^4x_4\end{array}
\]   
\caption{The diagram-reversing, weight-reversing bijection $f$.}\label{fig:f}
\end{figure}

A {\em standard Young row-strict composition tableau} (SYRT) of shape $\alpha// \beta$ is a SSYRT with each of the entries $1,2,\ldots, |\alpha//\beta|$ appearing exactly once.  There is a bijection between SYRT and saturated chains in $\YL_c$.  Given a saturated chain $\alpha^0<_{\YL_c}\alpha^1<_{\YL_c}\cdots <_{\YL_c} \alpha^k$ in $\YL_c$, construct an SYRT by placing the label $i$ in the cell of $\alpha^k//\alpha^0$ such that the cell is in $\alpha^i//\alpha^{i-1}$.  For example, the chain $(1)<_{\YL_c}(1,1)<_{\YL_c}(1,2)<_{\YL_c}(1,2,1)<_{\YL_c}(1,2,2)<_{\YL_c} (1,2,3)<_{\YL_c}(2,2,3)$ gives rise to the SYRT in Fig.~\ref{fig:SYRTbij}.
\begin{figure}
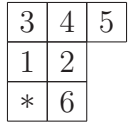

\[\tableau{ 3&4&5\\1&2\\*&6}\]
\caption{An SYRT of shape $(2,2,3)//(1)$.}\label{fig:SYRTbij}
\end{figure}

\section{Young row-strict quasisymmetric Schur functions}\label{sec:rsqf}

The Young quasisymmetric Schur functions introduced by Luoto, Mykytiuk, and van Willigenburg~\cite{LMvW13} are analogous to the quasisymmetric Schur functions~\cite{HLMvW09} since they retain many of the same properties and are described combinatorially through fillings of composition diagrams under slightly modified rules.  (Quasisymmetric Schur functions are generated by the weights of reverse composition tableaux while Young quasisymmetric Schur functions are generated by the weights of Young composition tableaux.)  This same analogy relates the polynomials introduced below to {\it row-strict quasisymmetric Schur functions} introduced in~\cite{MasRem10}.

\begin{definition}
Let $\alpha$ be a composition.  Then the Young row-strict quasisymmetric Schur function $\mathcal{R}_{\alpha}$ is given by $$\mathcal{R}_{\alpha}=\sum_T x^T$$ where the sum is over all Young row-strict composition tableaux (SSYRTs) $T$ of shape $\alpha$.  See Figure~\ref{SSYRTweights} for an an example.
\end{definition}

\begin{figure}
$$\tableau{2&3 \\ 2 \\ 1&2 \\ 1} \qquad \tableau{2&4 \\ 2 \\ 1&2 \\ 1} \qquad \tableau{3&4 \\ 2 \\ 1&2 \\ 1} \qquad \tableau{3&4 \\ 3 \\ 1&2 \\ 1} \qquad \tableau{3&4 \\ 3 \\ 1&3 \\ 1} \qquad \tableau{3&4 \\ 3 \\ 2&3 \\ 1} \qquad \tableau{3&4 \\ 3 \\ 2&3 \\ 2}$$
\begin{align*}\mathcal{R}&_{1212}(x_1,x_2,x_3,x_4)=\\&x_1^2x_2^3x_3 + x_1^2x_2^3x_4 + x_1^2x_2^2x_3x_4 + x_1^2x_2x_3^2x_4 + x_1^2x_3^3x_4 + x_1x_2x_3^3x_4 + x_2^2x_3^3x_4\end{align*}
\caption{The SSYRT that generate $\mathcal{R}_{(1,2,1,2)}(x_1,x_2,x_3,x_4)$.}
\label{SSYRTweights}
\end{figure}

In order to describe several important properties enjoyed by the \RSYQS s, we introduce a weight-preserving bijection between semi-standard Young tableaux and SSYRT.  The following map sends a row-strict semi-standard Young tableau $T$ to an SSYRT $\rho(T) = F$, described algorithmically, with an example shown in Fig.~\ref{fig:bijection}.

\begin{enumerate}
\item Place the entries from the leftmost column of $T$ into the first column of $F$ in weakly decreasing order from top to bottom.
\item Once the first $k-1$ columns of $T$ have been placed into $F$, place the entries from the $k^{th}$ column of $T$ into $F$, from smallest to largest.  Place an entry $e$ in the highest row $i$ such that $(i,k)$ does not already contain an entry from $T$ and the entry $(i, k-1)$ is strictly smaller than $e$.
\item Repeat until all entries in the $k^{th}$ column of $T$ have been placed into $F$.
\end{enumerate}

Note that during the second step of the procedure $\rho$, there is always an available cell where the entry $e$ can be placed.  This is due to the fact that in the row-strict semi-standard Young tableau the number of entries strictly smaller than $e$ in the column immediately to the left of that containing $e$ is greater than or equal to the number of entries in the column containing $e$ which are less than or equal to $e$.  Thus, $\rho$ is a well-defined procedure.

\begin{figure}
$$T= \tableau{6 \\ 2 & 4 & 5 & 7 \\ 1 & 3 & 5 & 6 \\  1 & 2 & 3 & 4 & 6 } \; \; \qquad \qquad \rho(T) = \tableau{6 \\ 2 & 3 & 5 & 6 \\  1 & 2 & 3 & 4 & 6 \\ 1 & 4 & 5  & 7}$$
\caption{The bijection $\rho$ applied to a row-strict semi-standard Young tableau $T$.}
\label{fig:bijection}
\end{figure}

\begin{lemma}{\label{lemma:bijection}}
The map $\rho$ is a weight-preserving bijection between the set of row-strict semi-standard Young tableaux of shape $\lambda$ and the set of SSYRT's of shape $\alpha$ where $\shape(\alpha) = \lambda$.
\end{lemma}

\begin{proof}
Let $T$ be a row-strict semi-standard Young tableau of shape $\lambda$.  Since $\rho$ preserves the column entries of $T$, $\rho$ is a weight-preserving map.  Since the number of entries in each column is preserved, the row lengths are preserved but possibly in a different order.  Therefore the shape $\alpha$ of $\rho(T)$ is a rearrangement of the partition $\lambda$ which describes the shape of $T$.  We next prove that $\rho(T)$ is indeed an SSYRT.

To see that $\rho(T)$ is an SSYRT, we must check the three conditions given in Definition~\ref{SSYRTdef}.  Conditions (1) and (2) are satisfied by construction.  To check condition (3), consider $i<j$ and the cell $(i,k+1)$ such that $\rho(T)(j,k) < \rho(T)(i,k+1)$.  We need to prove that $\rho(T)(j,k+1) \le \rho(T)(i,k+1)$.  To get a contradiction, assume that $\rho(T)(j,k+1) > \rho(T)(i,k+1)$.  Then the entry $\rho(T)(i,k+1)$ was inserted before the entry $\rho(T)(j,k+1)$.  Since $\rho(T)(j,k) < \rho(T)(i,k+1)$ and $\rho(T)(j,k+1)$ was not yet inserted, $\rho(T)(i,k+1)$ would have been inserted immediately to the right of the cell $(j,k)$.  This contradicts the fact that $\rho(T)(i,k+1)$ is in the $i^{th}$ row.  Therefore $\rho(T)(j,k+1) \le \rho(T)(i,k+1)$ and condition (3) is satisfied, so $\rho(T)$ is an SSYRT.

Finally, we prove that $\rho$ is a bijection by exhibiting its inverse.  Let $F$ be an SSYRT.  Then $\rho^{-1}(F)$ is given by placing the column entries from the $k^{th}$ column of $F$ into the $k^{th}$ column of $\rho^{-1}(F)$ so that they are weakly increasing from bottom to top.  By construction, $\rho^{-1}(F)$ is a row-strict semi-standard Young tableau.  It is also clear that $\rho^{-1} (\rho(T)) = T$.  We must prove that if $\rho^{-1}(F)=\rho^{-1}(G)$ for two SSYRTs $F$ and $G$, then $F=G$.  Assume, to the contrary, that $\rho^{-1}(F) = \rho^{-1}(G)$ and $F \not= G$.  Find the leftmost column, $k$, for which the arrangement of the column entries differs.  Note that $k>1$ by construction.  Find the highest cell $(i,k)$ in column $k$ such that $F(i,k) \not= G(i,k)$.  Then the entry $F(i,k)$ is found in a row (call it $r$) further down in the column in $G$.  So $F(i,k) = G(r,k)$.  If $F(i,k) < G(i,k)$ then cells $(i,k), (r,k),$ and $(i,k-1)$ in $G$ violate condition (3) and therefore $G$ is not an SSYRT.  If $F(i,k) > G(i,k)$, then the entry $G(i,k)$ is found in a row (call is $s$) further down in the column of $F$ and the cells $(i,k), (s,k)$ and $(i,k-1)$ in $F$ violate condition (3).  Therefore $F$ is not an SSYRT.  Therefore there is only one SSYRT mapping by $\rho^{-1}$ to a given row-strict semi-standard Young tableau $T$, which is precisely the SSYRT given by $\rho(T)$.  Thus $\rho$ is a bijection.
\end{proof}

The Young row-strict quasisymmetric Schur functions decompose the Schur functions in the following natural way.

\begin{theorem}{\label{thm:SchurSum}}
If $\lambda$ is an arbitrary partition, then $$s_{\lambda} = \sum_{\alpha: \shape(\alpha) = \lambda'}\mathcal{R}_{\alpha}.$$
\end{theorem}

\begin{proof}
This theorem is a direct consequence of Lemma~\ref{lemma:bijection} since the Schur function $s_{\lambda}$ is generated by the weights of all semi-standard Young tableaux whose transpose is a row-strict semi-standard Young tableau of shape $\lambda'$ and the right hand side sums the weights of all Young row-strict composition tableaux whose shape rearranges $\lambda'$.
\end{proof}

The {\it standardization} $st(F)$ of a Young row-strict composition tableau follows a procedure similar to that used to standardize a row-strict semi-standard Young tableau.  Once we provide an appropriate reading order on the entries in a Young row-strict composition tableau, we will describe the standardization procedure algorithmically.

\begin{definition}\label{def:ssyrtstd}
Read the entries of an SSYRT $F$ by column from right to left, so that the entries in each column are read from top to bottom, except the leftmost column, which is read from bottom to top.  This is called the {\it standard reading order} and the resulting word is called the {\it standard reading word of $F$} and is denoted $read(F)$.
\end{definition}

\begin{procedure}
To standardize a Young row-strict composition tableau $F$, use the following procedure.
\begin{enumerate}
\item Let $\beta=(\beta_1, \beta_2, \hdots , \beta_m)$ be the composition describing the content of $F$.
\item Replace the $j^{th}$ occurrence (in standard reading order) of the entry $1$ with the entry $j$.
\item For $i>1$, replace the $j^{th}$ occurrence (in standard reading order) of the entry $i$ with the entry $\displaystyle{(\sum_{k=1}^{i-1} \beta_k) + j}$.  
\end{enumerate}
\end{procedure}

The resulting filling, denoted $st(F)$, is the standardization of $F$.  See Figure \ref{fig:standardization} for an example of the standard reading word of an SSYRT $F$ and the standardization of $F$.

\begin{figure}
\[\begin{array}{cc}
F = \tableau{6 \\ 2 & 3 & 5 & 6 \\  1 & 2 & 3 & 4 & 6 \\ 1 & 4 & 5  & 7} & st(F)=\tableau{13\\4&6&9&12\\2&3&5&7&11\\1&8&10&14}\\\\
read(F)=66475353241126
\end{array}\]
\caption{The reading word and standardization of an SSYRT $F$ of shape $(4,5,4,1)$}
\label{fig:standardization}
\end{figure}

\begin{theorem}{\label{prop:standard}}
The map $\rho$ commutes with standardization in the sense that if $T$ is an arbitrary row-strict semi-standard Young tableau, then $(st(\rho(T)))=\rho(std(T))$ where $std$ is the ordinary standardization of a row-strict semi-\\standard Young tableau.
\end{theorem}

\begin{proof}
Note that the map $\rho$ does not change the set of entries in a given column.  Therefore it is enough to show that the set of entries in an arbitrary column of $std(T)$ is equal to the set of entries in the corresponding column of $st(\rho(T))$.

Let $T$ be a row-strict semi-standard Young tableau.  Consider two equal entries in cells $c_1$ and $c_2$ with $c_1$ appearing earlier in the reading order than $c_2$.  Since the rows of $T$ strictly increase from left to right along rows and weakly increase up columns, the cell $c_2$ must appear weakly to the left and strictly above $c_1$.  Thus, in the standardization of $T$, the entry in $c_1$ is replaced first.  If $c_1$ and $c_2$ are in distinct columns of $T$, then the entries appearing in $c_1$ and $c_2$ are in distinct columns of $\rho(T)$ and the entry in $c_1$ will appear in the reading order of $\rho(T)$ prior to the entry in $c_2$.  If $c_1$ and $c_2$ are in the same column of $T$, then the entry in $c_1$ is placed into $\rho(T)$ prior to the entry in $c_2$, and thus appears earlier in the reading order of $\rho(T)$ as well.  Thus the entries are replaced in $st(\rho(T))$ in the same order as they were in $std(T)$ and the column entries are therefore the same.  Therefore standardization commutes with $\rho$.
\end{proof}

To show that the Young row-strict quasisymmetric Schur functions are indeed quasisymmetric, we provide a method for writing a Young row-strict quasisymmetric Schur function as a positive sum of Gessel's fundamental quasisymmetric functions.  We need the following analogue of the descent set in order to describe this decomposition.

\begin{definition}
The {\it reverse descent set, $\hat{D}(T)$}, of a standard SSYRT $T$, is the subset of $[ n-1]$ consisting of all entries $i$ of $T$ such that $i+1$ appears strictly to the right of $i$ in $T$.
\end{definition}

\begin{proposition}{\label{fundDecomp}}
Let $\alpha, \beta$ be compositions.  Then $$\mathcal{R}_{\alpha} = \sum_{\beta} d_{\alpha \beta} F_{\beta}$$ where $d_{\alpha \beta}$ is equal to the number of standard Young row-strict composition tableaux $T$ of shape $\alpha$ such that $\comp(\hat{D}(T)) =  \beta$.
\end{proposition}

See Figure~\ref{fig:FundDecomp} for an example of this decomposition.

\begin{figure}
$$T_1= \tableau{2 & 3 & 5 \\ 1 & 4} \; \qquad T_2=\tableau{2 & 3 & 4 \\ 1 & 5}\; \qquad T_3=\tableau{3 & 4 & 5 \\ 1 & 2}$$

$$\hat{D}(T_1)=\{2,4\} \qquad \quad \hat{D}(T_2)=\{2,3\} \qquad \quad \hat{D}(T_3)=\{1,3,4\}$$

$$\mathcal{R}_{23}=F_{221}+F_{212}+F_{1211}$$
\caption{The decomposition of $\mathcal{R}_{23}$ into fundamental quasisymmetric functions.}
\label{fig:FundDecomp}
\end{figure}

We prove Proposition~\ref{fundDecomp} by using the decomposition of the row-strict quasisymmetric Schur functions into fundamental quasisymmetric functions~\cite{MasRem10} and applying the weight-reversing, diagram-reversing map $f$ from standard  SSRRT to standard SSYRT.

\begin{proposition}[Mason-Remmel~\cite{MasRem10}]
Let $\alpha$ and $\beta$ be compositions.  Then \begin{equation}\label{MasRemDecomp} \mathcal{RS}_{\alpha} = \sum_{\beta} d_{\alpha \beta} F_{\beta} \end{equation} where $d_{\alpha \beta}$ is equal to the number of standard row-strict composition tableaux $T$ of shape $\alpha$ and $\comp(D'(T))=\beta$.  Here $D'(T)$ is the set of all entries $i$ of $T$ such that $i+1$ appears strictly to the left of $i$ in $T$.
\end{proposition}

\begin{proof}{(of Proposition~\ref{fundDecomp})}

Recall that $\mathcal{RS}_\alpha$ is generated by the set of SSRRT of shape $\alpha$.  Applying the map $f$ to each of the SSRRT generating $\mathcal{RS}_\alpha$ yields an SSYRT of shape $rev(\alpha)$.  Thus, each monomial appearing on the left hand side of \eqref{MasRemDecomp} is sent to its reverse, and the same occurs on the right hand side.  Reversing the monomials sends $F_{\alpha}$ to $F_{rev(\alpha)}$, so we need to show that if $\comp(D'(T))=\beta$, then $\comp(\hat{D}(f(T)))=rev(\beta)$.

Consider an element $i \in D'(T)$.  Then $i+1$ appears strictly to the left of $i$ in $T$.  Reversing the entries in the SSYRT $T$ implies that $n-(i+1)+1=n-i$ appears strictly to the left of $n-i+1$ in $f(T)$.  This implies that $n-i+1$ appears strictly to the right of $n-i$ in $f(T)$, meaning $n-i$ is in the reverse descent set $\hat{D}(f(T))$.  To determine the composition $\comp(\hat{D}(f(T))$ obtained from the reverse descent set $\hat{D}(f(T))$, note that the $k^{th}$ part, $\beta_k$, of $\comp(D'(T))$ from the left, is equal to $i-j$, where $i$ is the $k^{th}$ smallest element of $D'(T)$ and $j$ is the $(k-1)^{th}$ smallest element of $D'(T)$ unless $k=1$ in which case $j=0$.  If $k \not=1$, then $n-i$ is the $k^{th}$ largest element in $\hat{D}(f(T))$ and $n-j$ is the $(k-1)^{th}$ largest element in $\hat{D}(f(T))$.  Therefore the $k^{th}$ part (from the right) of $\hat{D}(f(T))$ is equal to $(n-j) - (n-i) =i-j$.  If $k=1$, then $j=0$ and $n-i$ is the last part of $\comp(\hat{D}(f(T)))$.  Therefore the composition obtained from $\hat{D}(f(T))$ is the reverse of the composition obtained from $D'(T)$ as desired.\end{proof}

\begin{cor}
Every \RSYQS \, is quasisymmetric, since it can be written as a positive sum of quasisymmetric functions.
\end{cor}

In fact, the \RSYQS s form a basis for all quasisymmetric functions.

\begin{theorem}
The set $\{ \mathcal{R}_{\alpha} ( x_1, \hdots , x_k) | \alpha \models n \, {\rm and} \, k \ge n \}$ forms a $\mathbb{Z}$-basis for \newline $QSym_n ( x_1, x_2, \hdots , x_k)$. \end{theorem}

\begin{proof}
Let ${ \bf x_k}$ denote $(x_1, \ldots, x_k)$.  It is enough to prove that the transition matrix from Young row-strict quasisymmetric Schur functions to fundamental quasisymmetric functions is upper uni-triangular with respect to a certain order on compositions.  Order the compositions indexing the Young row-strict quasisymmetric Schur functions by any linear extension of the dominance order on the underlying partitions.  (That is, a linear extension of the partial order where $\alpha < \beta$ if and only if $\shape(\alpha)  < \shape(\beta)$ under dominance.)  Similarly, order the compositions indexing the fundamental quasisymmetric functions by the same order but on their complementary compositions.

Fix a positive integer $n$ and a composition $\alpha=(\alpha_1, \alpha_2, \hdots , \alpha_{\ell(\alpha)})$ of $n$.  Consider a summand $F_{\beta}({\bf x_k})$ appearing in $\mathcal{R}_{\alpha}({\bf x_k})$ and consider the composition $\tilde{\beta}$ complementary to $\beta$.  Since $k \ge n$, we know that such an $F_{\beta}({\bf x_k})$ exists.  We claim that $\shape(\tilde{\beta}) \le \shape(\alpha)$ and, moreover, if $\shape(\tilde{\beta}) = \shape(\alpha)$, then $\tilde{\beta} = \alpha$.

Since $F_{\beta}({\bf x_k})$ appears in $\mathcal{R}_{\alpha}({\bf x_k})$, there must exist a standard Young row-strict composition tableau $F$ of shape $\alpha$ with reverse descent set \\$D=\{b_1, b_2, \hdots, b_k\}$ such that $\comp(D)=\beta$.  Each entry $i$ in $D$ must appear in $F$ strictly to the left of $i+1$.  We note however there are examples of standard Young row strict composition tableaux where $i+1$ can appear in the same row, a row above $i$, or a row below $i$.  Nevertheless, the consecutive entries in $D$ appear in $F$ as ``horizontal strips'' in that no two cells can lie in the same column.  Each such horizontal strip, together with the entry one greater than the largest entry in the strip, corresponds to a part of $\tilde{\beta}$ since $\tilde{\beta}$ arises from the complement of $D$.  This implies that the sum of the largest $j$ parts of $\tilde{\beta}$ must be less than or equal to the sum of the largest $j$ parts of $\alpha$.  That is, $j$ horizontal strips cannot cover more cells than the number of cells in the largest $j$ rows of $F$. Therefore $\shape(\tilde{\beta}) \le \shape(\alpha)$.

Assume that $\shape(\tilde{\beta}) = \shape(\alpha)$.  Since the length of $\tilde{\beta}$ is equal to the length of $\alpha$, each cell in the leftmost column of $F$ must be the start of a horizontal strip. Moreover, each horizontal strip must have cells in a block of consecutive columns. That is, the horizontal strip corresponding to the largest part of $\tilde{\beta}$,  $\lambda(\tilde{\beta})_1$, has the same size as the largest row in $\alpha$ and hence must contain one cell in each column even though those cells may not lie in the same row. But then the horizontal strip corresponding the second largest part of $\tilde{\beta}$ has the same size as the second largest part of $\alpha$ so it must have one cell in each column up to the rightmost column in the second largest row  of $\alpha$, etc.  Next we claim that row numbers of the cells in any horizontal strip must weakly increase, reading from left to right. That is, suppose that $x-1$ and $x$ are consecutive entries of a horizontal strip such that $x$ appears in a lower row than $x-1$ in $F$. Then the triple rule (condition (3) in Definition~\ref{SSYRTdef}) implies that the cell immediately to the right of the cell containing $x-1$ in the column containing $x$ must contain an entry less than or equal to $x$.  Since $F$ is standard, no entry can appear twice, and hence this entry must be less than $x-1$.  But this contradicts the fact that the row entries are increasing from left to right.

Now consider the set of elements in the descent set $1,2, \ldots, \tilde{\beta}_1-1$ of $F$ corresponding to the first part of $\tilde{\beta}$. We know that  $1,2 \ldots, \tilde{\beta}_1$ forms a horizontal strip $h$ in $F$ and $1$ must be in the leftmost column.  It must be the case that $\tilde{\beta}_1$ is in the first row since otherwise the first column would not be strictly increasing reading from bottom to top. Since the column numbers in $h$ must strictly increase reading from left to right, it follows that $1,\hdots , \tilde{\beta}_1$ occupy consecutive cells in the first row of $F$. Since no element from a higher strip can appear in this row, this means that the first row of $F$ has size $\tilde{\beta}_1$. That is, $\alpha_1 = \tilde{\beta}_1$.  Iterating this argument implies that $\tilde{\beta}_i = \alpha_i$ for all $i$, and hence $\tilde{\beta} = \alpha$ as claimed.  Therefore the transition matrix is upper uni-triangular and the proof is complete.
\end{proof}

\section{Skew row-strict quasisymmetric Schur functions}\label{sec:skewrs}

\subsection{The Hopf algebra of quasisymmetric functions}

We briefly discuss results from Hopf algebras that will be pertinent to the definition of $\mathcal{R}_{\alpha//\beta}$ and refer the reader to any of the standard references~\cite{Abe80, HGK10, Swe69} for more information and detailed definitions.

An {\it associative algebra $A$ over a field $k$} is a $k$-vector space with an associative bilinear map $m:  A \otimes A \longrightarrow A$ and {\it unit} $u:k \longrightarrow A$ which sends $1 \in k$ to the two-sided multiplicative identity element $1_A \in A$.  A {\it co-associative algebra $C$} is a $k$-vector space $C$ with a $k$-linear map $\Delta: C \longrightarrow C \otimes C$ (called {\it co-multiplication}) and {\it co-unit} $\epsilon: C \longrightarrow k$ such that 

\begin{enumerate}
\item $(1_C \otimes \Delta) \circ \Delta = (\Delta \otimes 1_C) \circ \Delta$
\item $(1_C \otimes \epsilon) \circ \Delta = 1_C = (\epsilon \otimes 1_C) \circ \Delta$.
\end{enumerate}
In the following, we use {\it Sweedler notation} for the coproduct and abbreviate $$\displaystyle{\Delta(c) = \sum_i c_{(1)}^{(i)} \otimes c_{(2)}^{(i)}} \qquad {\rm by} \qquad \Delta(c) = \sum_{(c)} c_{(1)} \otimes c_{(2)}.$$

Given a coalgebra $C$ and an algebra $A$, the $k$-linear maps $Hom(C,A)$ can be endowed with an associative algebra structure called the {\it convolution algebra} which sends $f,g \in Hom(C,A)$ to $f \star g$ defined by $(f \star g)(c) = \sum f(c_1)g(c_2)$.  A {\it Hopf Algebra} is a bialgebra $A$ which contains an two-sided inverse $S$ (called an {\it antipode}) under $\star$ for the identity map $1_A$.  (Note that every connected, graded bialgebra is a Hopf algebra~\cite{Ehr96}.)

Recall that the {\it dual space $V^{\star}$} of a finite-dimensional $k$-vector space $V$ is obtained by reversing $k$-linear maps so that $V^{\star}:= Hom(V,k)$.  If each $V_n$ in a graded vector space $\displaystyle{V= \oplus_{n \ge 0} V_n}$ is finite-dimensional, then $V$ is said to be of {\it finite type}.  Since the dual of a finite-type algebra is a coalgbera and vice-versa, the dual $H^{\star}= \oplus_{n \ge 0} H_n^{\star}$ of a finite-type Hopf algebra $\displaystyle{ H= \oplus_{n \ge 0} H_n}$ is a Hopf algebra called the {\it Hopf dual} of $H$.  Therefore there exists a non-degenerate bilinear form $\langle \cdot , \cdot \rangle : H \otimes H^{\star} \longrightarrow k$ such that for any basis $\{ B_i \}_{i \in I}$ (for some indexing set $i$) of $H_n$ and its dual basis $\{ D_i \}_{ i \in I}$, we have $\langle B_i, D_j \rangle = \delta_{ij}$.  This duality allows us to define structure constants and skew elements ~\cite{LLS09}

\begin{enumerate}
\item $$B_i \cdot B_j = \sum_h a_{i,j}^h B_h \iff \Delta D_h = \sum_{i,j} a_{i,j}^h D_i \otimes D_j := \sum_j D_{i/j} \otimes D_j$$
\item $$D_i \cdot D_j = \sum_h b_{i,j}^h D_h \iff \Delta B_h = \sum_{i,j} b_{i,j}^h B_i \otimes B_j := \sum_j B_{i/j} \otimes B_j.$$
\end{enumerate}

In $QSym$, the coproduct is given by 
\begin{equation} \label{eq:coprod}
\Delta F_\alpha = \sum_{\substack{\beta\gamma = \alpha, \text{ or }\\ \beta\odot \gamma = \alpha}} F_\beta \otimes F_\gamma
\end{equation}
where $\beta\gamma$ denotes concatenation and $\beta\odot \gamma$ denotes {\it almost concatenation}.  That is, if $\beta = (\beta_1, \ldots, \beta_k)$ and $\gamma = (\gamma_1, \ldots, \gamma_m)$, then $\beta\gamma = (\beta_1, \ldots, \beta_k, \gamma_1, \ldots, \gamma_m)$ while $\beta\odot \gamma = (\beta_1, \ldots, \beta_{k-1}, \beta_k+\gamma_1, \gamma_2, \ldots, \gamma_m)$.

We use this to define the {\it skew Young row-strict quasisymmetric Schur functions}.   

\begin{definition}\label{def:skew}
Given a composition $\alpha$,
\[\Delta \mathcal{R}_\alpha = \sum_{\beta} \mathcal{R}_\beta \otimes \mathcal{R}_{\alpha//\beta}.\]
\end{definition}

\subsection{Combinatorial formulas for skew Young row-strict quasisymmetric Schur functions}

Recall that Schur functions can be computed explicitly using the following combinatorial rule~\cite{Ful97}:  $$s_{\lambda / \mu} = \sum_{T \in SSYT(\lambda / \mu)} x^{cont(T)} = \sum_{T \in SYT(\lambda/ \mu)} F_{Des(T)},$$ where $F_{S}$ is the Gessel fundamental quasisymmetric function indexed by the set $S$.  We provide an analogous combinatorial formula for the skew Young row-strict quasisymmetric Schur functions.

\begin{theorem}\label{thm:skew}
Let $\alpha // \beta$ be a skew composition.  Then $\mathcal{R}_{\alpha // \beta}$ is the sum over all SSYRT $T$ of shape $\alpha // \beta$ of the weights $x^T$; i.e. $$\mathcal{R}_{\alpha // \beta} = \sum_{T \in \SSYRT(\alpha // \beta)} x^T = \sum_{T \in \SYRT(\alpha // \beta)} F_{D'(T)}.$$  
\end{theorem}

Before proving Theorem~\ref{thm:skew}, we introduce several concepts and lemmas which will play an important role in the proof.  

\begin{definition}
If $T$ is a filling of shape $\alpha//\beta$, then $g_i$ is the map that adds $i$ to each entry in $T$.
\end{definition}

Let $T$ be a SYRT of shape $\alpha \vDash n$.  Then define $\mho_i(T)$ to be the filling obtained from $T$ by considering the cells with labels less than or equal to $i$.  Similarly, let $\Omega_i(T)$ be the skew filling obtained from $T$ by considering the cells with labels greater than $n-i$ and standardizing.  Note that $T = \mho_i(T) \cup (g_i (\Omega_{n-i}(T)))$, as seen in Fig.~\ref{fig:decomp}.

\begin{figure}
\[\begin{array}{ccc}
T = \tableau{7&8&10\\6\\2&3&4&9\\1&5} \quad&\quad \mho_6(T) = \tableau{*&*&*\\6\\2&3&4&*\\1&5}\quad &\quad \Omega_4(T) = \tableau{1&2&4\\*\\*&*&*&3\\*&*}\\\\
\hat{D}(T) = \{2,3,7\} & \hat{D}(\mho_6(T)) = \{2,3\} &\hat{D}(\Omega_4(T)) = \{1\}\\
\comp(\hat{D}(T)) = (2,1,4,3) & \comp(\hat{D}(\mho_6(T))) = (2,1,3) & \comp(\hat{D}(\Omega_4(T))) = (1,3)
\end{array}
\]
\caption{Decomposing $T$ using $\mho_i(T)$ and $\Omega_{n-i}(T)$.}\label{fig:decomp}
\end{figure}

\begin{lemma}\label{lem:glue}
Let $\beta \subset \gamma$ be a composition contained in $\gamma$ with $|\beta| = i$.  If $T(\gamma/ / \beta)$ is a skew SSYRT of shape $\gamma / / \beta$ and $T(\beta)$ is an SSYRT of shape $\beta$, then $T(\beta) \cup g_i (T(\gamma // \beta))$ is an SSYRT.
\end{lemma}

\begin{proof}
We need to check that $T':=T(\beta) \cup g_i (T(\gamma // \beta))$ satisfies the three conditions needed to be an SSYRT.  The first is satisfied by construction since given a row $r$ in $T'$, the entries from row $r$ of $T(\beta)$ are strictly increasing from left to right, the entries from row $r$ of $ g_i (T(\gamma // \beta))$ are strictly increasing from left to right, and the last entry in row $r$ of $T(\beta)$ is less than the first entry in row $r$ of $ g_i (T(\gamma // \beta))$.

The second condition is satisfied by construction since the entries in the leftmost column of $ g_i (T(\gamma // \beta))$ weakly decrease from top to bottom, the entries in the leftmost column of $T(\beta)$ weakly decrease from top to bottom, and the lowest entry in the leftmost column of $ g_i (T(\gamma // \beta))$ is greater than the highest entry in the leftmost column of $T(\beta)$.

Finally, to check the third condition, consider a triple of cells $a:=T'(j,k), $ $b:=T'(j,k+1), \, c:=T'(i,k+1)$.  If $a<c$, then there are several cases we must consider.

\underline{Case 1:  $a,c \in T(\beta)$:}  If $b \in T(\beta)$ then we are done since the entries $a,b,c$ must satisfy the triple rule in $T(\beta)$.  But if $b \notin T(\beta)$, then $b$ would be considered to be equal to $\infty$, causing $a,b,c$ to violate the triple condition in $\beta$, a contradiction.  Therefore $b \in T(\beta)$ and $a,b,c$ satisfy the triple condition.

\underline{Case 2:  $a \in T(\beta), c \in  g_i (T(\gamma // \beta))$:}  If $b \in T(\beta)$ then we are done since all entries from $T(\beta)$ are less than all entries from $ g_i (T(\gamma // \beta))$, so $b \le c$.  If $b \notin T(\beta)$, then $b \le c$ since $a$ is considered to be zero in $ g_i (T(\gamma // \beta))$, so $b \le c$ in $T'$.

\underline{Case 3:  $a, c \in  g_i (T(\gamma // \beta))$:}  Then $b \in  g_i (T(\gamma // \beta))$, so the triple $a,b,c$ satisfies the triple condition in $ g_i (T(\gamma // \beta))$ so the proof is complete.
\end{proof}

\begin{lemma}\label{lem:comp}
Let $T$ be a SYRT of shape $\alpha$.  Then for any $i$ ($0 \leq i\leq |\alpha|$), if $i \in \hat{D}(T)$, then
\[\comp(\hat{D}(T)) = \comp(\hat{D}(\mho_i(T))) \cdot \comp(\hat{D}(\Omega_{n-i}(T))),\]
and if $i \notin \hat{D}(T)$, 
\[\comp(\hat{D}(T)) = \comp(\hat{D}(\mho_i(T))) \odot \comp(\hat{D}(\Omega_{n-i}(T))).\]
\end{lemma}

\begin{proof} Suppose $\hat{D}(T) = \{a_1,a_2,\ldots,a_k\}$ and let $0\leq i\leq |\alpha|$.  Then, if $i \in \hat{D}(T)$, $i = a_m$ for some $m$.  Then $\hat{D}(\mho_i(T)) = \{a_j : a_j<i\}$ and \[\comp(\hat{D}(\mho_i(T))) = (a_1, a_2-a_1, \ldots, i- a_{m-1}).\]  Similarly, $\hat{D}(\Omega_{n-i}(T)) = \{ a_j-i:a_j>i\}$ and 
\begin{align*}
\comp(\hat{D}(\Omega_{n-i}(T)))&  =  (a_{m+1}-i, a_{m+2}-i - (a_{m+1}-i), \ldots, n-i - (a_k-i))\\
& = (a_{m+1}-a_m, a_{m+2}-a_{m+1}, \ldots, n-a_k).
\end{align*}
Thus, 
\[
\comp(\hat{D}(\mho_i(T)))\cdot \comp(\hat{D}(\Omega_{n-i}(T)))  = (a_1, a_2-a_1, \ldots, n-a_k) = \comp(\hat{D}(T)).\]
Now suppose $ i \notin \hat{D}(T)$.  If $a_m>i$ for all $m$, then $\hat{D}(\mho_i(T)) = \emptyset$, $\hat{D}(\Omega_{n-i}(T) = \{ a_1-i, \ldots, a_k-i\}$, $\comp(\hat{D}(\mho_i(T))) = (i)$, and $\comp(\hat{D}(\Omega_{n-i}(T))) = (a_1-i, a_2-a_1, \ldots, n-a_k)$, so $\comp(\hat{D}(T)) = \comp(\hat{D}(\mho_i(T))) \odot \comp(\hat{D}(\Omega_{n-i}(T))) $.  Similarly, if $a_{m-1}<i<a_m$ for some $m$, then $\comp(\hat{D}(\mho_i(T))) = (a_1, a_2-a_1, \ldots, i-a_{m-1})$ while $\comp(\hat{D}(\Omega_{n-i}(T))) = (a_m - i, a_{m+1}-a_m, \ldots, n-a_k)$.  Thus, $\comp(\hat{D}(T)) = \comp(\hat{D}(\mho_i(T))) \odot \comp(\hat{D}(\Omega_{n-i}(T))) $. \end{proof}

We are now ready to prove Theorem \ref{thm:skew}.

\begin{proof}
Let $\alpha$ be a composition of size $n$.  We first expand $\Delta \mathcal{R}_\alpha$ in terms of the fundamental quasisymmetric functions, using \eqref{eq:coprod} and Lemma~\ref{lem:comp}. 
\begin{align}
\Delta \mathcal{R}_\alpha & = \sum_{T\in SYRT(\alpha)} \Delta F_{\comp(\hat{D}(T))} \notag \\
&= \sum_{T\in SYRT(\alpha)} \left( \sum_{\substack{\beta\cdot\gamma=\comp(\hat{D}(T))\text{ or}\\\beta\odot \gamma=\comp(\hat{D}(T))}} F_\beta \otimes F_\gamma  \right)\notag\\\notag\\
&=\sum_{T\in SYRT(\alpha)}\left( \sum_{i=0}^n F_{\comp(\hat{D}(\mho_i(T)))} \otimes F_{\comp(\hat{D}(\Omega_{n-i}(T)))}\right)~\label{eqn:Fsplit}.
\end{align}

Equating the right hand side of Definition~\ref{def:skew} with the right hand side of (\ref{eqn:Fsplit}) gives the following equality.

\begin{align*}~\label{eqn:funds}
\sum_{\delta} \mathcal{R}_\delta \otimes \mathcal{R}_{\alpha//\delta} & =\sum_{T\in SYRT(\alpha)}\left( \sum_{i=0}^n F_{\comp(\hat{D}(\mho_i(T)))} \otimes F_{\comp(\hat{D}(\Omega_{n-i}(T)))}\right),
\end{align*}
where the $\delta$ on the left hand side ranges over all compositions contained in $\alpha$.  Expanding $\mathcal{R}_\delta$ in terms of the fundamental quasisymmetric functions allows us to write

\begin{align*}
\sum_{\delta} \left(\sum_{U\in SYRT(\delta)}F_{\comp(\hat{D}(U))}\right)& \otimes \mathcal{R}_{\alpha//\delta}  =\\&\sum_{T\in SYRT(\alpha)}\left( \sum_{i=0}^n F_{\comp(\hat{D}(\mho_i(T)))} \otimes F_{\comp(\hat{D}(\Omega_{n-i}(T)))}\right).
\end{align*}

By Lemma~\ref{lem:glue}, for each $U \in SYRT(\delta)$, there exists a $T\in SYRT(\alpha)$ and $i$ with $0\leq i\leq n$ such that $U = \mho_i(T)$.  Note that $\Omega_{n-i}(T)$ will have shape $\alpha//\delta$.  Thus, we obtain
\begin{align*}
\mathcal{R}_{\alpha//\delta} & = \sum_{T\in SYRT(\alpha//\delta)} F_{\comp(\hat{D}(T))}.
\end{align*}
\end{proof}

\section{Properties of skew row-strict quasisymmetric Schur functions}\label{sec:LRrule}

Every skew Schur function can be decomposed into a positive sum of skew row-strict quasisymmetric Schur functions.  This decomposition is analogous to the decomposition of the Schur functions into quasisymmetric Schur functions.

\begin{theorem}
Each skew Schur function is a positive sum of skew row-strict quasisymmetric Schur functions given by $$s_{\lambda' / \mu'} = \sum_{\shape(\alpha // \beta) = \lambda / \mu} \mathcal{R}_{\alpha // \beta}.$$
\end{theorem}

\begin{proof}
We exhibit a weight-preserving bijection, $h$, between the set of all column-strict tableaux of shape $\lambda' / \mu'$ and the set of all row-strict composition tableaux whose shape rearranges to $\lambda/\mu$.  See Fig.~\ref{fig:skewschur} for an example of the bijection.

Given a column-strict tableau $S$ of shape $\lambda' / \mu'$, consider the entries contained in $\mu'$ to be zeros, represented by stars in the diagram.  Take the conjugate of $S$ to produce a row-strict tableau $r(S)$ of shape $\lambda/\mu$.  Next map the entries in the leftmost column of $r(S)$ to the leftmost column of $h(S)=T$ by placing them in weakly decreasing order.  Map each set of column entries from $r(S)$ into the corresponding column of $h(S)$ by the following process:

\begin{enumerate}
\item Begin with the smallest entry, $a_1$ in the set.
\item Map $a_1$ to the highest available cell that is immediately to the right of an entry strictly smaller than $a_1$.  If $a_1=0$, map it to the highest available cell immediately to the right of a zero.
\item Repeat with the next smallest entry, noting that a cell is {\it available} if no entry has already been placed in this cell.
\item Continue until all entries from this column have been placed, and then repeat with each of the remaining columns.
\end{enumerate} 

We must show that this process produces a row-strict composition tableau.  The first two conditions are satisfied by construction, so we must check the third condition.  Consider two cells $T(j,k)$ and $T(i,k+1)$ such that $j>i$, $(i,k+1) \in \alpha // \beta$, and $T(j,k) < T(i,k+1)$.  Let $T(j,k)=b$, $T(i,k+1)=a$, and $T(j,k+1)=c$.  Then the cells are situated as shown, where $a > b$:
$$\tableau{b & c \\ \\ & a}.$$
We must prove that $T(j,k+1) \le T(i,k+1)$, or in other words, that $c \le a$.  Assume, to get a contradiction, that $c > a$.  Then $a$ would be inserted into its column before $c$.  But then the cell immediately to the right of $b$ would have been available, and therefore $a$ would have been placed in that cell rather than farther down in the column.  Therefore this configuration would not have occurred.

The inverse map, $h^{-1}$ is given by arranging the entries from each column of a row-strict composition tableau $U$ so that they are weakly increasing from top to bottom.  We must prove that the row entries are strictly increasing.  Argue by contradiction.  Assume there exists a row whose entries are not strictly increasing, and choose the leftmost column, $C$, such that an entry in $C$ is smaller than the entry immediately to its left and the highest occurrence (call it row $R$) within this column.  Call this entry $x$.  Then column $C-1$ contains only $R-1$ entries which are less than $x$ while column $C$ contains $R$ entries less than or equal to $x$.  This contradicts the fact that the rows of $U$ are strictly increasing.
\end{proof}

\begin{figure}
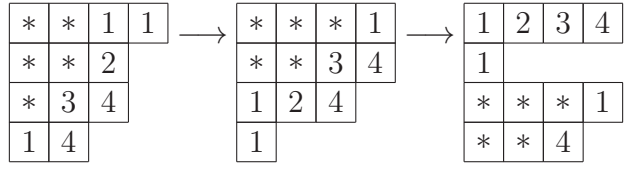

$$\tableau{* & * & 1 & 1 \\ * & * & 2 \\ * & 3 & 4 \\ 1 & 4} \longrightarrow \tableau{ * & * & * & 1 \\ * & * & 3 & 4 \\ 1  & 2 & 4 \\ 1} \longrightarrow \tableau{1 & 2 & 3 & 4 \\ 1 \\ * & * & * & 1 \\ * & * & 4}$$
\caption{The bijection $h$ from SSYT to SSRYT.}\label{fig:skewschur}
\end{figure}

\subsection{A multiplication rule}

In this section we give a Littlewood-Richardson style rule for multiplying $\mathcal{R}_\alpha \cdot s_\lambda$.  We show that this rule is equivalent to the rule given by Ferreira in \cite{ferreira} for $\mathcal{RS}_\alpha \cdot s_\lambda$.  We begin by reviewing the RSK algorithms for both SSRRT and SSYRT, and then present Ferreira's Littlewood-Richardson rule for $\mathcal{RS}_\alpha \cdot s_\lambda$, followed by the Littlewood-Richardson rule for $\mathcal{R}_\alpha \cdot s_\lambda$.

\begin{algorithm} \label{alg:insertssrrt}~\cite{MasRem10}
Given $T\in \SSRRT(\alpha)$, we insert $x$ into $T$, denoted $T\xleftarrow{R} x$, in the following way:
\begin{enumerate}
\item For each cell $T(i,\alpha_i)$, add a cell $T(i,\alpha_i+1) = 0$.
\item Read down each column of ${T}$, starting from the rightmost column and moving left. This is the {\em reading order} for ${T}$.

\item Find the first entry ${T}(i,k+1)$ in reading order such that ${T}(i,k+1)\leq x$ and ${T}(i,k)>x$ and $k \neq 1$.  
\begin{enumerate}
\item If $T(i,k+1)=0$, set $T(i,k+1) = x$ and the algorithm terminates.
\item If $T(i,k+1)=y$, set $T(i,k+1)=x$ and continue scanning cells using $x=y$ starting at cell $(i-1,k+1)$.  Again, we say that ``$x$ bumps $y$.''
\item If there is no such entry, then insert $x$ into the first column, creating a new row of length 1, in between the unique pair ${T}(i,1)$ and ${T}(i+1,1)$ such that ${T}(i,1)>x\geq {T}(i+1,1)$ and terminate the insertion.  If $x \geq {T}(i,1)$ for all $i$, insert $x$ at the bottom of the first column, creating a new row of length 1, and terminate the insertion.  
\end{enumerate} 
\item Continue until the insertion terminates.
\end{enumerate}
\end{algorithm}

\begin{algorithm} \label{alg:insertssyrt}
Given $F\in \SSYRT(\alpha)$, we insert $x$ into $F$, denoted $F\leftarrow x$, in the following way:
\begin{enumerate}
\item For each cell $F(i,\alpha_i)$, add a cell $F(i,\alpha_i+1) = \infty$.

\item Read down each column of ${F}$, starting from the rightmost column and moving left. This is the {\em reading order} for ${F}$.

\item Find the first entry ${F}(i,k+1)$ in reading order such that ${F}(i,k+1)\geq x$ and ${F}(i,k)<x$ and $k \neq 1$. 
\begin{enumerate}
\item If $F(i,k+1)=\infty$, set $F(i,k+1):=x$ and the algorithm terminates.
\item If $F(i,k+1)=y$, set $F(i,k+1):= x$ and continue scanning cells using $x:=y$ starting at cell $(i-1,k+1)$.   We say that ``$x$ bumps $y$.''
\item If there is no such entry, then insert $x$ into the first column, creating a new row of length 1, in between the unique pair ${F}(i,1)$ and ${F}(i+1,1)$ such that ${F}(i,1)<x\leq {F}(i+1,1)$ and terminate the insertion.  If $x\leq {F}(i,1)$ for all $i$, insert $x$ at the bottom of the first column, creating a new row of length 1, and terminate the insertion.  
\end{enumerate}


\item Continue until the insertion terminates.
\end{enumerate}
\end{algorithm}

In Fig.~\ref{fig:ssyrtbump} the insertion path is highlighted in bold when 3 is inserted into $F$, denoted $F\leftarrow 3$.

\begin{figure}
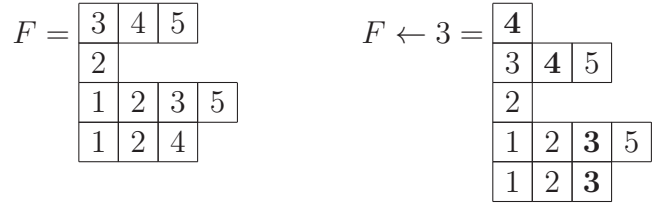

\[F=\tableau{3&4&5\\2\\1&2&3&5\\1&2&4}\qquad\qquad F\leftarrow 3 = \tableau{\boldsymbol{4}\\3&\boldsymbol{4} & 5\\ 2\\ 1&2&\boldsymbol{3}&5\\1&2&\boldsymbol{3}}\]
\caption{The result of Algorithm~\ref{alg:insertssyrt}, with insertion path in bold.}\label{fig:ssyrtbump}
\end{figure}

A word $w= w_1w_2\ldots w_n$ with $\max\{w_i\} = m$ is a {\em lattice word} if for every prefix of $w$, there are at least as many $i$'s as $(i+1)$'s for each $1\leq i <m$.  Note that such a word must start with a 1.  A {\em reverse lattice word} is a word $v=v_1v_2\ldots v_n$ with $\max\{v_i\}= m$ with the property that for every prefix of $v$, for all $i \le m$, there are at least as many $i$'s as ($i-1$)'s.  A reverse lattice word is called {\em regular} if it contains at least one 1.  Given a (skew) SSYRT $T$, the {\em column word} of $T$, denoted $\col(T)$, is the word obtained by reading each column from bottom to top, starting with the leftmost column and moving right.  The column reading word of a (skew) SSRRT is obtained by this same procedure.  Note that this is different from the standard reading order defined previously.

Given a weak composition $\gamma$, let $\gamma^+$ denote the composition obtained from $\gamma$ by removing all parts of size 0.  Then, if $\alpha \subseteq \beta$, denote by $\beta/\alpha$ any skew shape $\beta//\gamma$ where $\gamma$ is a weak composition such that $\gamma^+ =\alpha$.  

Consider an arbitrary filling $F$ of shape $\beta/\gamma$.  Rather than a single triple rule, we will require two types of triples, A and B, as seen in Fig.~\ref{fig:ABtriples}.  A {\em type A triple} is a set of entries $F(i,k-1)=a, F(i, k)=b,$ and $F(j,k)=c$ where row $i$ occurs above row $j$ and $\beta_i\geq \beta_j$.  A {\em type B triple} is a set of entries $F(i, k-1)=c, F(j,k-1)=a,$ and $F(j,k)=b$ where row $i$ occurs above row $j$ and $\beta_i<\beta_j$.  Note that in English convention, $i<j$, while in French convention $i>j$, but the triples will have the same shape in either case. 
\begin{figure}
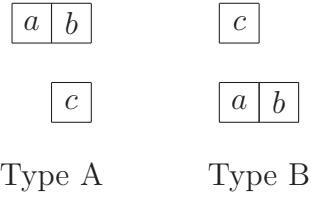

\[\begin{array}{ccc}
\tableau{a&b\\\\&c} &\quad & \tableau{c\\\\a&b}\\\\
\text{ Type A } && \text{ Type B }\end{array}\]
\caption{Type A triples require that $\beta_i\geq \beta_j$ where row $i$ is above row $j$, and Type B triples require $\beta_i<\beta_j$ where row $i$ is above row $j$.}\label{fig:ABtriples}
\end{figure}

In \cite{ferreira} a {\em reverse Littlewood-Richardson skew row-strict composition tableau} is defined to be a filling $L$ of a diagram of a skew shape $\beta/\gamma$ where $\beta$ and $\alpha$ are strong compositions and $\gamma$ is a weak composition with $\gamma^+ = \alpha$, and the diagram also has a column $0$ filled with $\infty$'s, with the following properties:
\begin{enumerate}
\item Each row strictly decreases when read left to right.
\item The column reading word is a regular reverse lattice word. 
\item Set $L(i,j) = \infty$ for all $(i,j) \in \gamma$. Each triple (type A or B) must satisfy $c\leq b<a$ or $b<a\leq c$, including triples containing cells in column 0.  See Fig.~\ref{fig:ABtriples} to see the types of triples.
\end{enumerate}

\begin{theorem}[\cite{ferreira}] Let $s_\lambda$ be the Schur function indexed by the partition $\lambda$, and let $\alpha$ be a composition.  Then 
\[\mathcal{RS}_\alpha \cdot s_\lambda = \sum_{\beta} C_{\alpha,\lambda}^\beta \mathcal{RS}_\beta\] 
where $C_{\alpha,\lambda}^\beta$ is the number of reverse Littlewood-Richardson skew SSYRT of shape $\beta/\alpha$ and content $rev(\lambda)$.
\end{theorem}

\begin{definition}\label{def:LRcomp} Let $\alpha$ and $\beta$ be compositions with $\alpha \subseteq \beta$.  Let $\gamma$ be a weak composition satisfying $\gamma^+ = \alpha$.  A {\em Littlewood-Richardson skew row-strict composition tableau} $S$ of shape $\beta/\alpha$ is a filling of a diagram of skew shape $\beta//\gamma$, with column 0 filled with 0's, such that 
\begin{enumerate}
\item Each row strictly increases when read left to right.
\item The column reading word of $S$ is a lattice word.
\item Let $S(i, j) = 0$ for all $(i,j) \in \gamma$.  Each triple (type A or B) must satisfy $a<b\leq c$ or $c\leq a <b$, including triples containing cells in column $0$.
\end{enumerate}
\end{definition}

We use the weight-reversing, shape-reversing bijection $f:\SSRRT\rightarrow $\\$\SSYRT$ defined in \S \ref{sec:comptab} to prove the following analogue of the Littlewood-Richardson rule. 
 
\begin{theorem}\label{thm:LRcomp} Let $\lambda$ be a partition and $\alpha$ a composition.  Then 
\[\mathcal{R}_\alpha\cdot s_\lambda = \sum_\beta D_{\alpha,\lambda}^\beta \mathcal{R}_\beta\] 
where $D_{\alpha,\lambda}^\beta$ is the number of Littlewood-Richardson skew SSYRT of shape $\beta//\alpha$ and content $\lambda$.
\end{theorem}

We need the following lemmas in order to prove Theorem~\ref{thm:LRcomp}.

\begin{lemma}\label{lem:commute} Given $T$, an SSRRT of shape $\alpha\vDash n$, $x$ a positive integer and $x^* = n-x+1$, $ f(T)\leftarrow x^*  = f(T\xleftarrow{R}x)$.
\end{lemma}

\begin{proof} Let $T$ be a  SSRRT of shape $\alpha$.  Then $f(T)$ is a SSYRT of shape $rev(\alpha)$.  We show that if, during the insertion process $T \leftarrow x$, the entry in cell $a$ of $T$ is bumped by an entry $y$, then the entry in cell $a$ of $f(T)$ is bumped by $y*=n-y+1$.  First consider the initial entry, $x$, inserted into $T$ and say this entry bumps the entry in cell $a$ of $T$.  Let $\hat{a}$ denote the cell immediately to the left of $a$.  Then if $T(a)$ is bumped, $T(a)\leq x$ and $T(\hat{a})>x$.  In the reading order, all cells $b$ that occur prior to $a$ must have $T(b)<x$ with $T(\hat{b})\leq x$.  In $f(T)$, for cell $b$ earlier in the reading order than $a$, note that $f(T)(b)>x^*$ and $f(T)(\hat{b}) \geq x^*$, thus there is no location prior to cell $a$ for $x^*$ to be inserted.  However, $f(T)(a)\geq x^*$ and $f(T)(\hat{a})<x^*$, so $x^*$ will bump the label in cell $a$. 

Now assume that the statement is true for all entries in the bumping process up to an entry $y$.  If $y$ bumps the entry in cell $b$, then the cell $\hat{b}$ immediately to the left of $b$ must contain an entry strictly greater than $b$.  So in $f(T)$, the entry in cell $\hat{b}$ must be strictly less than the entry in cell $b$.  So $y^*$ will bump the entry in cell $b$ of $f(T)$ if $y^*$ does not bump an entry before this cell.  But if $y^*$ bumps an entry in $f(T)$ before $a$ in reading order, then $y$ would bump the entry in that cell in $T$.  This cannot happen since $y$ bumps the entry in cell $b$.  Therefore the bumps are the same and hence the resulting diagram $f(T) \leftarrow x^*$ is the diagram obtained by $f(T \leftarrow x)$.
\end{proof}

\begin{lemma}\label{lem:LRbij} If $T$ is a reverse Littlewood-Richardson skew row-strict composition tableau of shape $\beta/\alpha$ and content $rev(\lambda)$, then $f(T)$ is a Littlewood-Richardson skew row-strict composition tableau of shape $rev(\beta)/rev(\alpha)$ and content $\lambda$.\end{lemma}

\begin{proof}
Let $T$ be a reverse Littlewood-Richardson skew row-strict composition tableau of shape $\beta/\alpha$ and content $rev(\lambda)$.  Then $f(T)$ has shape $rev(\beta)/rev(\alpha)$.  In $T$, the number of entries $i$ is given by $\lambda_{n-i+1}$ since the content of $T$ is $rev(\lambda)$.  Thus, in $f(T)$, the number of entries $i$ is given by $\lambda_i$, so the content of $f(T)$ is $\lambda$.  Conditions (1) and (3) for reverse Littlewood-Richardson skew row-strict composition tableaux are true since we know that $f(T)$ is an SSYRT.  Therefore we must prove Condition (2); that is that the column reading word of $f(T)$ is a lattice word.

Let $\col(T) = w_1w_2\ldots w_k$.  Then $\col(f(T)) = v_1v_2\ldots v_k$ where $v_i = n-w_i+1$.  Let $v_1\ldots, v_j$ be a prefix of $\col(f(T))$ and let $x_i$ be the number of $i$'s in the prefix.  Then $n-i+1$ appears $x_i$ times in the corresponding prefix $w_1\ldots w_j$ in $\col(T)$, and then, since $\col(T)$ is a reverse regular lattice word, $n-i$ appears at most $x_i$ times in the prefix $w_1\ldots w_j$.  Thus, there are at most $x_i$ appearances of $i+1$ in $v_1\ldots v_j$.  Therefore $\col(f(T))$ is a lattice word.

Finally, we consider the type A and B triples.  Consider a type A or B triple in $T$ with entries $a, b, c$ as arranged in Fig.~\ref{fig:ABtriples}.  If $c\leq b<a$, then the corresponding triple in $f(T)$ has $n-c+1\geq n-b+1>n-a+1$, and thus satisfies the triple condition in Definition~\ref{def:LRcomp}.  Similarly, if $b<a\leq c$ in $T$, then $n-b+1>n-a+1\geq n-c+1$ in $f(T)$ and the triple condition is satisfied.\end{proof}

We are now ready to prove Theorem~\ref{thm:LRcomp}.  
\begin{proof}  Let $\lambda$ be a partition and $\alpha$ a composition.  Then $D_{\alpha,\lambda}^\beta = C_{rev(\alpha),\lambda}^{rev(\beta)}$ for each $\beta$ with $\alpha \subseteq \beta$ by Lemma~\ref{lem:LRbij}.  Thus,

\begin{align*}
\mathcal{R}_\alpha\cdot s_\lambda &= \mathcal{RS}_{rev(\alpha)} \cdot s_\lambda\\
&=\sum_{rev(\beta)} C^{rev(\beta)}_{rev(\alpha),\lambda} \mathcal{RS}_{rev(\alpha)}\\
&=\sum_\beta D^{\beta}_{\alpha,\lambda} \mathcal{R}_\beta.
\end{align*}\end{proof}

By applying the reversing bijection $f$ and the bijection in \cite{ferreira}, we can describe a bijection from pairs $(U,S)$ where  $U$ is an SSYRT of shape $\alpha$ and $S$ is a row-strict semi-standard Young tableau of shape $\lambda^t$ to a pair $(V,T)$ where $V$ is a SSYRT of shape $\beta$ and $T$ is a Littlewood-Richardson skew SSYRT of shape $\beta/\alpha$ and content $\lambda$.  See figure Fig.~\ref{fig:LRrule} for an example of the following procedure.

First, define $S_\lambda$ to be the row-strict semi-standard Young tableau of shape $\lambda^t$ with all entries 1 in column 1, all entries 2 in column 2, etc.  Map the pair $(S, S_\lambda)$ to a double word by placing the entries read by column, from top to bottom, left to right, so that the entries from $S_\lambda$ form the top word and the entries of $S$ form the bottom word. Then to construct $V$ and $T$, start with $U$ and a tableau of shape $\alpha$ filled with *'s. Insert the labels from the bottom row of the double word into $U$ in order from left to right.  After each insertion, add a cell to the star tableau in the location where a cell was added to $U$.  Fill this new cell with the entry in the top row corresponding to the entry inserted into $U$.  The result of the insertion into $U$ is $V$ and the result of the insertion into the star tableau is $T$ as shown in Fig.~\ref{fig:LRrule}.  The equivalence of this bijection to the bijection in \cite{ferreira} follows immediately from Lemmas~\ref{lem:commute} and \ref{lem:LRbij}.   

\begin{figure}
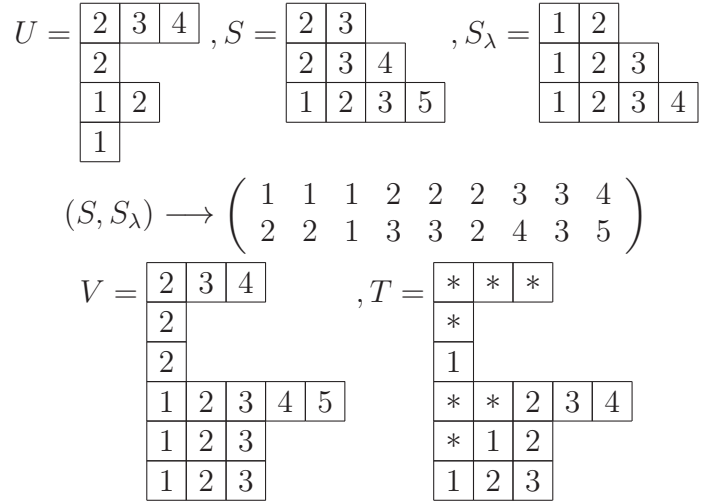

\[U = \tableau{2&3&4\\2\\1&2\\1}\,\,, S=\tableau{ 2&3\\2&3&4\\1&2&3&5}\,, S_\lambda = \tableau{1&2\\1&2&3\\1&2&3&4 }\]
\[ \left(S, S_\lambda\right) \longrightarrow \left(\begin{array}{ccccccccc} 1&1&1&2&2&2&3&3&4\\2&2&1&3&3&2&4&3&5\end{array}\right)\]
\[V=\tableau{2&3&4\\2\\2\\1&2&3&4&5\\1&2&3\\1&2&3}\,\, , T = \tableau{*&*&*\\*\\1\\*&*&2&3&4\\*&1&2\\1&2&3}\]
\caption{Using insertion to obtain L-R skew SSYRT.}\label{fig:LRrule}
\end{figure}

\subsection{A conjugation-like map for composition tableaux}
As noted in Section~\ref{sec:intro}, traditional notions of conjugation simply do not work for composition tableaux.  The result of reflecting a composition diagram over the main diagonal might not be a composition diagram, so a new operation is needed.  In this section we describe an analog to the function $\phi$ introduced in \cite{MasRem10}.  The function $\phi$ takes {\em semi-standard reverse composition tableaux} (SSRCT) to SSRRT.  Given a SSRCT $T$ of shape $\alpha$ with $\boldsymbol{\lambda}(\alpha) = \lambda$, $\phi(T)$ is a SSRRT of shape $\gamma$ with $\boldsymbol{\lambda}(\gamma) = \lambda'$.  That is, the composition $\gamma$ has the underlying partition $\lambda'$, the conjugate of $\lambda$.  To construct $\phi(T)$ column by column, first take the largest entry in each column of $T$ and arrange this set of entries in weakly increasing order from top to bottom to obtain the first column of $\phi(T)$. Then take the next largest entry in each column of $T$ and insert the entries into the second column of $\phi(T)$ by starting with the largest entry and placing it in the highest position such that the entry immediately to its left is strictly greater. Continue likewise.  This procedure is well-defined since the entries in a given column of $T$ are distinct and the rows of $T$ are weakly increasing.

This map can be quickly modified to a function taking {\em semi-standard Young composition tableaux} (SSYCT) as defined in \cite{LMvW13} to SSYRT.  We define $\tilde{\phi}:\SSYCT\rightarrow \SSYRT$ as follows: given a SSYCT $U$ of shape $\alpha$ with $\boldsymbol{\lambda}(\alpha) = \lambda$, $\tilde{\phi}(U)$ is a SSYRT of shape $\gamma$ with $\boldsymbol{\lambda}(\gamma) = \lambda'$.  We construct $\tilde{\phi}(U)$ column by column, by taking the {\em smallest} entry in each column of $U$ and placing these entries in weakly increasing order (from bottom to top) in the first column of $\tilde{\phi}(U)$. To construct the second column, take the second smallest entry of each column of $U$ and insert these entries into the second column of $\tilde{\phi}(U)$ starting with the smallest entry and placing it in the highest position such that the entry immediately to its left is strictly smaller. Continue likewise.  Again, this map is well-defined because of the row and column restrictions on $U$.  An example of $\phi$ and $\tilde{\phi}$ can be seen in Fig.~\ref{fig:phi}.

\begin{lemma}\label{lem:phicommute}
Let $T$ be an SSRCT of shape $\alpha$.  Then $(f\circ \phi)(T) = (\tilde{\phi}\circ f)(T)$.
\end{lemma}

\begin{proof}
Note that the largest entry in a column of $T$ will become the smallest entry in that column of $f(T)$.  Therefore the entries in the leftmost column of $(\tilde{\phi} \circ f)(T)$ are obtained by reversing the entries in the leftmost column of $\phi(T)$ in order since $\phi$ arranges these entries in weakly decreasing order and $\tilde{\phi}$ arranges these entries in weakly increasing order.  Similarly, the entries in the second column of $\phi(T)$ are the reversals of the entries in the second column of $(\tilde{\phi}\circ f)(T)$ and since the maps $\phi$ and $\tilde{\phi}$ involve the same process with the roles of increasing and decreasing reversed, the second column of $(\tilde{\phi}\circ f)(T)$ is obtained by reversing the entries in the second column of $\phi(T)$.  Similarly, each column of $(\tilde{\phi}\circ f)(T)$ is obtained by reversing the entries in the corresponding column of $\phi(T)$, which is precisely the procedure used in the function $f$, so $(f\circ \phi)(T) = (\tilde{\phi}\circ f)(T)$.
\end{proof}

\begin{figure}
\[ \begin{array}{ccc}
T=\tableau{4&2&1\\*&*&4&4\\*&*&3\\*&1}&\xrightarrow{\phi} & \phi(T) = \tableau{4&3&2&1\\4\\*&*&*&4\\*&*&1}\\\\
f\downarrow && f\downarrow\\\\

f(T)=U= \tableau{1&3&4\\*&*&1&1\\*&*&2\\*&4} & \xrightarrow{\tilde{\phi}} & \tilde{\phi}(U) = \tableau{1&2&3&4\\1\\*&*&*&1\\*&*&4}
\end{array}\]
\caption{The result of $\phi$ and $\tilde{\phi}$.}\label{fig:phi}

\end{figure}
This function $\tilde{\phi}$ establishes that 
\[ \sum_{\substack{\alpha\\\boldsymbol{\lambda}(\alpha) = \lambda}}S_\alpha = \sum_{\substack{\beta\\\boldsymbol{\lambda}(\beta) = \lambda'}} \mathcal{R}_\beta,\] where $S_\alpha$ is the Young quasisymmetric Schur function defined in Appendix~\ref{appendix}.
\section{Connections and Future Directions}

The {\it omega operator} $\omega$ sends the Schur function $s_{\lambda}$ to the Schur function $s_{\lambda'}$, where $\lambda'$ is the transpose of the partition $\lambda$.  This operator can be extended to quasisymmetric functions via the antipode map~\cite{MalReu95}.  The Hopf algebra $QSym$ has antipode $S$ defined by $$S(F_{\alpha}) = (-1)^{| \alpha |} F_{w(\alpha)},$$ where $w(\alpha)$ is the composition corresponding to $\tilde{\alpha}$, the set complement of the set corresponding to $\alpha$~\cite{MalReu95}.  The omega operator on quasisymmetric functions is defined by $\omega(F_{\alpha}) = (-1)^{|\alpha|} S(F_{\alpha}),$ meaning $\omega(F_{\alpha}) = F_{w(\alpha)}$.

This operator, applied to the Young row-strict quasisymmetric Schur functions, mimics the conjugation appearing in the symmetric function case.  That is, since $\omega(s_{\lambda}) = s_{\lambda'}$, and $s_{\lambda} = \tilde{s}_{\lambda}$ where $\tilde{s}_{\lambda}$ is the function generated by row-strict Young diagrams of shape $\lambda$, it is natural to expect that the quasisymmetric refinement of the omega operator would send a Young quasisymmetric Schur function to a Young row-strict quasisymmetric Schur function, which is indeed the case.

\begin{theorem}
The Young quasisymmetric Schur functions are conjugate to the Young row-strict quasisymmetric Schur functions in the following sense: $$\omega(S_{\alpha}(x_1, \hdots , x_n)) =\mathcal{R}_{\alpha}(x_n, \hdots , x_1).$$
\end{theorem}

\begin{proof}

Expand $S_{\alpha}$ in terms of the fundamental quasisymmetric functions and apply the omega operator to obtain:

\begin{eqnarray*}
\omega(S_{\alpha}(x_1, \hdots , x_n))  & = & \sum_{\beta} d_{\alpha \beta} \omega(F_{\beta}(x_1, \hdots , x_n)) \\
 & = & \sum_{\beta} d_{\alpha \beta} F_{w({\beta})}(x_1, \hdots , x_n) \\
 & = & \sum_{\beta} d_{\alpha \beta} F_{rev(w({\beta}))}(x_n, \hdots , x_1) \\
 & = & \mathcal{R}_{\alpha}(x_n, \hdots , x_1).
\end{eqnarray*}

Here, the first equation is derived from the expansion of the Young quasisymmetric Schur functions in terms of the fundamental quasisymmetric Schur functions~\cite{LMvW13}.  In particular, $d_{\alpha \beta}$ is the number of SYCT of shape $\alpha$ whose descent set (the values $i$ such that $i+1$ appears weakly to the left of $i$) is $\beta$.  The second equation is obtained by applying the quasisymmetric extension of the omega operator to the fundamental quasisymmetric functions appearing in the expansion of the Young quasisymmetric Schur functions.  Note that $w(\beta)$ selects all values $i$ such that $i$ is not in the descent set of the corresponding SYCT.  This is precisely the values $i$ such that $i+1$ appears strictly to the right of $i$.  The third equality is direct from the definition of the fundamental quasisymmetric functions, and the fourth equality is a direct consequence of the expansion of the Young row-strict quasisymmetric Schur functions in terms of the fundamental quasisymmetric functions, since $i+1$ appearing strictly to the right of $i$ is equivalent to $n-i+1$ appearing strictly to the right of $n-(i+1)+1=n-i$ when the variables of the SYCT are reversed.  This means $n-i$ is in the descent set of the corresponding SYRT as desired.
\end{proof}

As discussed in Section~\ref{sec:rsqf}, the Young row-strict quasisymmetric Schur functions can be obtained from the row-strict quasisymmetric Schur functions in the same way that the Young quasisymmetric Schur functions are obtained from the quasisymmetric Schur functions.  We use the Young versions rather than the version originally introduced in~\cite{MasRem10} since the generating diagrams for the Young version map more naturally to semi-standard Young tableaux, whereas the original generating diagrams map to reverse semi-standard Young tableaux.  The quasisymmetric Schur functions can be derived by summing Demazure atoms~\cite{HLMvW09} over all weak compositions which collapse to the same composition.  The Demazure atoms are specializations of nonsymmetric Macdonald polynomials to $q=t=0$~\cite{HHL06, Mas08}.  Although the Young row-strict quasisymmetric functions are in fact a completely different basis for quasisymmetric functions than the row-strict quasisymmetric functions, they behave similarly and exhibit all of the same combinatorial properties.  It is natural to expect that the Young quasisymmetric Schur functions can also be obtained as sums of specializations of a certain Macdonald polynomial variant, although the precise details of this derivation have yet to be worked out.

It would be helpful to understand the Young row-strict quasisymmetric Schur functions as characters of some representation.  In ~\cite{TvW} Tewari and van Willigenburg develop a representation theoretical interpretation of the quasisymmetric Schur functions $\mathcal{QS}_\alpha$ where \[\mathcal{QS}_\alpha = \sum_{T \in SSRCT(\alpha)} x^T\] is a sum over semistandard {\em reverse} column-strict composition tableaux (SSRCT).  They show that there is a 0-Hecke module $\boldsymbol{S}_\alpha$ whose quasisymmetric characteristic is $\mathcal{QS}_\alpha$.  One future project is to identify the appropriate 0-Hecke module with quasisymmetric characteristic $\mathcal{R}_\alpha$.

\newpage
\appendix
\section{The four basic types of composition tableaux}\label{appendix}

\begin{tabular}{|p{1.5in}|p{1.5in}|p{3in}|}
\hline
{\bf Object Name} & {\bf Function Generated}&{\bf Basic Properties}\\\hline\hline
Semi-standard reverse~row-strict composition tableau (SSRRT)& $\mathcal{RS}_\alpha $ \newline
Row-strict quasisymmetric Schur function & \begin{enumerate}
\item Diagram in English convention.
\item Row entries strictly decrease left to right.
\item The leftmost column weakly increases from top to bottom.
\item For cells arranged as in Fig.~\ref{fig:triple}, if $c< a$, then $c\leq b$.
\end{enumerate} \\\hline
Semi-standard Young row-strict composition tableau (SSYRT) & $\mathcal{R}_\alpha$ \newline Row-strict Young quasisymmetric Schur function& \begin{enumerate}
\item Diagram in French convention.
\item Row entries strictly increase left to right.
\item The leftmost column weakly decreases from top to bottom.
\item For cells arranged as in Fig.~\ref{fig:triple}, if $a<c$, then $b\leq c$.
\end{enumerate}  \\\hline
Semi-standard reverse composition tableau (SSRCT) & $\mathcal{QS}_\alpha$ \newline Quasisym-\newline metric Schur function& \begin{enumerate}
\item Diagram in English convention.
\item Row entries weakly decrease left to right.
\item The leftmost column strictly increases from top to bottom.
\item For cells arranged as in Fig.~\ref{fig:triple}, if $c\leq a$, then $c< b$.
\end{enumerate}\\\hline
Semi-standard Young composition tableau (SSYCT) & $\mathcal{S}_\alpha $\newline Young quasisymmetric Schur function&\begin{enumerate}
\item Diagram in French convention.
\item Row entries weakly increase left to right.
\item The leftmost column strictly decreases from top to bottom.
\item For cells arranged as in Fig.~\ref{fig:triple}, if $a\leq c$, then $b< c$.
\end{enumerate}  \\\hline
\end{tabular}

\begin{figure}[h]
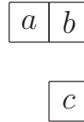

\[\tableau{ a&b\\\\ & c}\]
\caption{Triple configuration for composition tableaux.}\label{fig:triple}
\end{figure}

\begin{acknowledgements*}
Both authors were partially supported by a Wake Forest University Collaboration Pilot Grant.
\end{acknowledgements*}

\bibliographystyle{spmpsci}      

\begin{thebibliography}{10}
\providecommand{\url}[1]{{#1}}
\providecommand{\urlprefix}{URL }
\expandafter\ifx\csname urlstyle\endcsname\relax
  \providecommand{\doi}[1]{DOI~\discretionary{}{}{}#1}\else
  \providecommand{\doi}{DOI~\discretionary{}{}{}\begingroup
  \urlstyle{rm}\Url}\fi

\bibitem{Abe80}
Abe, E.: Hopf algebras, \emph{Cambridge Tracts in Mathematics}, vol.~74.
\newblock Cambridge University Press, Cambridge-New York (1980).
\newblock Translated from the Japanese by Hisae Kinoshita and Hiroko Tanaka

\bibitem{ABS06}
Aguiar, M., Bergeron, N., Sottile, F.: Combinatorial {H}opf algebras and
  generalized {D}ehn-{S}ommerville relations.
\newblock Compos. Math. \textbf{142}(1), 1--30 (2006).
\newblock \doi{10.1112/S0010437X0500165X}.
\newblock \urlprefix\url{http://dx.doi.org/10.1112/S0010437X0500165X}

\bibitem{Ehr96}
Ehrenborg, R.: On posets and {H}opf algebras.
\newblock Advances in Mathematics \textbf{119}(1), 1--25 (1996)

\bibitem{ferreira}
Ferreira, J.: A {L}ittlewood-{R}ichardson type rule for row-strict
  quasisymmetric {S}chur functions.
\newblock Disc. Math. Theor. Comp. Sci. Proc. \textbf{AO}, 329--338 (2011)

\bibitem{Ful97}
Fulton, W.: Young tableaux, \emph{London Mathematical Society Student Texts},
  vol.~35.
\newblock Cambridge University Press, Cambridge (1997).
\newblock With applications to representation theory and geometry

\bibitem{Ges84}
Gessel, I.: Multipartite p-partitions and inner products of skew {S}chur
  functions.
\newblock Contemp. Math \textbf{34}, 289--301 (1984)

\bibitem{GesReu93}
Gessel, I.M., Reutenauer, C.: Counting permutations with given cycle structure
  and descent set.
\newblock Journal of Combinatorial Theory, Series A \textbf{64}(2), 189--215
  (1993).
\newblock \doi{http://dx.doi.org/10.1016/0097-3165(93)90095-P}.
\newblock
  \urlprefix\url{http://www.sciencedirect.com/science/article/pii/009731659390095P}

\bibitem{HHL06}
Haglund, J., Haiman, M., Loehr, N.: A combinatorial formula for non-symmetric
  {M}acdonald polynomials.
  \newblock Amer. J. Math. \textbf{130}(2), 359--383  (2008).

\bibitem{HLMvW09}
Haglund, J., Luoto, K., Mason, S., van Willigenburg, S.: Quasisymmetric {S}chur
  functions.
\newblock J. Combin. Theory Ser. A \textbf{118}(2), 463--490 (2011).
\newblock \doi{10.1016/j.jcta.2009.11.002}.
\newblock \urlprefix\url{http://dx.doi.org/10.1016/j.jcta.2009.11.002}

\bibitem{HGK10}
Hazewinkel, M., Gubareni, N., Kirichenko, V.V.: Algebras, rings and modules,
  \emph{Mathematical Surveys and Monographs}, vol. 168.
\newblock American Mathematical Society, Providence, RI (2010).
\newblock \doi{10.1090/surv/168}.
\newblock \urlprefix\url{http://dx.doi.org/10.1090/surv/168}.
\newblock Lie algebras and Hopf algebras

\bibitem{Hiv00}
Hivert, F.: Hecke algebras, difference operators, and quasi-symmetric
  functions.
\newblock Advances in Mathematics \textbf{155}(2), 181--238 (2000).
\newblock \doi{http://dx.doi.org/10.1006/aima.1999.1901}.
\newblock
  \urlprefix\url{http://www.sciencedirect.com/science/article/pii/S0001870899919011}

\bibitem{LLS09}
Lam, T., Lauve, A., Sottile, F.: Skew {L}ittlewood-{R}ichardson rules from
  {H}opf algebras.
\newblock Arxiv preprint arXiv:0908.3714  (2009)

\bibitem{LMvW13}
Luoto, K., Mykytiuk, S., van Willigenburg, S.: An introduction to
  quasisymmetric {S}chur functions.
\newblock Springer Briefs in Mathematics. Springer, New York (2013).
\newblock \doi{10.1007/978-1-4614-7300-8}.
\newblock \urlprefix\url{http://dx.doi.org/10.1007/978-1-4614-7300-8}.
\newblock Hopf algebras, quasisymmetric functions, and Young composition
  tableaux

\bibitem{MalReu95}
Malvenuto, C., Reutenauer, C.: Duality between quasi-symmetric functions and
  the {S}olomon descent algebra.
\newblock J. Algebra \textbf{177}(3), 967--982 (1995).
\newblock \doi{10.1006/jabr.1995.1336}.
\newblock \urlprefix\url{http://dx.doi.org/10.1006/jabr.1995.1336}

\bibitem{Mas08}
Mason, S.: A decomposition of {S}chur functions and an analogue of the
  {R}obinson-{S}chensted-{K}nuth algorithm.
\newblock S\'eminaire Lotharingien de Combinatoire \textbf{57}(B57e) (2008).
\newblock
  \urlprefix\url{http://www.citebase.org/abstract?id=oai:arXiv.org:math/0604430}

\bibitem{MasRem10}
Mason, S., Remmel, J.: Row-strict quasisymmetric schur functions.
\newblock Annals of Combinatorics \textbf{18}(1), 127--148 (2014)

\bibitem{Sch01}
Schur, I.: \"{U}ber eine klasse von matrizen, die sich einer gegebenen matrix
  zuordnen lassen.
\newblock Ph.D. thesis, Berlin (1901)

\bibitem{Sta84}
Stanley, R.P.: On the number of reduced decompositions of elements of coxeter
  groups.
\newblock European Journal of Combinatorics \textbf{5}(4), 359--372 (1984).
\newblock \doi{http://dx.doi.org/10.1016/S0195-6698(84)80039-6}.
\newblock
  \urlprefix\url{http://www.sciencedirect.com/science/article/pii/S0195669884800396}

\bibitem{Sta99}
Stanley, R.P.: Enumerative combinatorics. {V}ol. 2, \emph{Cambridge Studies in
  Advanced Mathematics}, vol.~62.
\newblock Cambridge University Press, Cambridge (1999).
\newblock With a foreword by Gian-Carlo Rota and appendix 1 by Sergey Fomin

\bibitem{Swe69}
Sweedler, M.E.: Hopf algebras.
\newblock Mathematics Lecture Note Series. W. A. Benjamin, Inc., New York
  (1969)

\bibitem{TvW}
Tewari, V., Willigenburg, S.: Quasisymmetric schur functions and modules of the
  0-hecke algebra.
\newblock DMTCS Proceedings \textbf{0}(01) (2014).
\newblock
  \urlprefix\url{http://www.dmtcs.org/dmtcs-ojs/index.php/proceedings/article/view/dmAT0111}

\bibitem{Ugl00}
Uglov, D.: Skew {S}chur functions and {Y}angian actions on irreducible
  integrable modules of {$\widehat{\mathfrak{g l}}_N$}.
\newblock Ann. Comb. \textbf{4}(3-4), 383--400 (2000).
\newblock \doi{10.1007/PL00001287}.
\newblock \urlprefix\url{http://dx.doi.org/10.1007/PL00001287}.
\newblock Conference on Combinatorics and Physics (Los Alamos, NM, 1998)

\end{thebibliography}

\end{document}